\documentclass[12pt]{amsart}

\usepackage{amsmath}
\usepackage{amssymb}
\usepackage{amsthm}
\usepackage[applemac]{inputenc}
\usepackage{enumitem}

\usepackage{tikz}
\usetikzlibrary{snakes}
\usetikzlibrary{arrows}

\newcommand{\R}{{\mathbb R}}
\newcommand{\Z}{{\mathbb Z}}
\newcommand{\N}{{\mathbb N}}

\newcommand{\Q}{{\mathbb Q}}

\newcommand{\CCC}{\mathcal{C}}
\newcommand{\EEE}{\mathcal{E}}
\newcommand{\HHH}{\mathcal{H}}

\newcommand{\PPP}{\mathcal{P}}
\renewcommand{\PPP}{P}

\newcommand{\XXX}{\mathcal{X}}
\newcommand{\YYY}{\mathcal{Y}}

\newcommand{\SL}{\operatorname{SL}}
\newcommand{\GL}{\operatorname{GL}}
\newcommand{\Aff}{\operatorname{Aff}}
\newcommand{\Lsurf}{\operatorname{L}}
\newcommand{\dist}{\operatorname{dist}}

\newcommand{\e}{\varepsilon}
\newcommand{\Va}{V_{a,b}}
\newcommand{\Vp}{V_{a',b'}}

\newcommand{\alert}[1]%
  {\par\medskip\par\hspace*{-1cm}%
  \fbox{\footnotesize\parbox{\linewidth}{#1}}\par\medskip\par}

\newcommand\mat[4]{%
  \bigl( \begin{smallmatrix} #1&#2\\ #3&#4 
  \end{smallmatrix}\bigr)
  }
\newcommand\vct[2]{\begin{pmatrix} #1\\#2\end{pmatrix}}

\newtheorem{theo}{Theorem}

\newtheorem{coro}[theo]{Corollary}
\newtheorem{lemma}[theo]{Lemma}

\theoremstyle{definition}
\newtheorem{claim}[theo]{Claim}
\newtheorem*{claim*}{Claim}

\theoremstyle{remark}
\newtheorem*{defi}{Definition}
\newtheorem{remk}[theo]{Remark}
\newenvironment{proofof}[1]{\noindent {\bf Proof of #1.}}{\hfill\qed\par}

\begin{document}

% ------------------------------------------------------------------
% ------------------------------------------------------------------

\title{The Ehrenfest wind-tree model: periodic directions, recurrence,
diffusion}

\author{Pascal Hubert}
\address{Laboratoire d'Analyse, Topologie et Probabilités, 
Case cour A,
Faculté des Sciences de  Saint-Jerôme,
Avenue Escadrille Normandie-Niemen,
13397 Marseille Cedex 20, France.}
\email{hubert@cmi.univ-mrs.fr}
\urladdr{http://www.cmi.univ-mrs.fr/{\lower.7ex\hbox{\~{}}}hubert/} \date{}

\author{Samuel Lelièvre}
\address{Laboratoire de mathématique d'Orsay,
UMR 8628 CNRS / Université Paris-Sud 11,
Bâtiment 425, Campus Orsay-vallée,
91405 Orsay cedex, France.}
\email{samuel.lelievre@math.u-psud.fr}
\urladdr{http://www.math.u-psud.fr/{\lower.7ex\hbox{\~{}}}lelievre/} \date{}

\author{Serge Troubetzkoy}
\address{Centre de physique théorique\\
Federation de Recherches des Unites de Mathematique de Marseille\\
Institut de mathématiques de Luminy and\\ 
Université de la Méditerranée\\ 
Luminy, Case 907, F-13288 Marseille Cedex 9, France}
\email{troubetz@iml.univ-mrs.fr}
\urladdr{http://iml.univ-mrs.fr/{\lower.7ex\hbox{\~{}}}troubetz/} \date{}
 
\date{October 2, 2009}

\begin{abstract}
We study periodic wind-tree models, unbounded planar billiards with 
periodically located rectangular obstacles.  For a class of
rational parameters we show the existence of completely periodic
directions, and recurrence; for another class of rational parameters,
there are directions in which all trajectories escape, and we
prove a rate of escape for almost all directions.  These results 
extend to a dense $G_\delta$ of parameters.
\end{abstract}

\keywords{Billiards, periodic orbits, recurrence, diffusion, 
square-tiled surfaces}
\subjclass[2000]{30F30, 37E35, 37A40}

\maketitle \markboth{Pascal Hubert, Samuel Lelièvre, Serge
Troubetzkoy}{The Ehrenfest wind-tree model: periodic directions,
recurrence, diffusion}

% ------------------------------------------------------------------

\section{Introduction}

In 1912 Paul and Tatiana Ehrenfest proposed the wind-tree model of
diffusion in order to study the statistical interpretation of the
second law of thermodynamics and the applicability of the Boltzmann
equation \cite{EhEh}.  In the Ehrenfest wind-tree model, a point
(``wind'') particle moves on the plane and collides with the usual law
of geometric optics with randomly placed fixed square scatterers
(``tree'').

In this paper, we study periodic versions of the wind-tree model: the
scatterers are identical rectangular obstacles located periodically
along a square lattice on the plane, one obstacle centered at each
lattice point.  We call the subset of the plane obtained by removing
the obstacles the billiard table (see some pictures in the Appendix),
even though it is a non compact space.  Hardy and Weber \cite{HaWe}
have studied the periodic model, they proved recurrence and abnormal
diffusion of the billiard flow for special dimensions of the obstacles
and for very special directions, using results on skew products above
rotations.  In the general periodic case, the situation is much more
difficult and, since this nice result, there has been no progress.  In
fact very few results are known for linear flows on translation
surfaces of infinite area, or for billiards in irrational polygons
from which they also arise (see however \cite{DDL}, \cite{GuTr},
\cite{Ho}, \cite{HuWe}, \cite{Tr}, \cite{Tr1}, \cite{Tr2}).

\subsection{Statement of results}

Without loss of generality, we assume that the lattice is the standard
$\Z^2$-lattice.  We denote by $a$ and $b$ the dimensions of the
rectangular obstacles, and by $T_{a,b}$ the corresponding billiard
table.  The set of billiard tables under study is hence parametrized
by $(a,b)$ in the noncompact parameter space $\XXX =
(0,1)^2$.

We define a dense subset of parameter values:
\begin{align*}
\EEE & = 
\bigl\{\, (a,b) = (p/q,r/s) \in \Q \times \Q :
\\
& \phantom{={}}\quad
(p,q) = (r,s) =1,\quad 0<p<q,\quad 0<r<s,
\\[-3pt]
& \phantom{={}}\quad
p,r \textrm{ odd,}\quad q,s \textrm{ even}
\,\bigr\}.
\end{align*}

Given a flow $\Phi$ acting on a measured topological space
$(\Omega,\mu)$, a point $x \in \Omega$ is \emph{recurrent} for $\Phi$
if for every neighborhood $U$ of $x$ and any $T_0 > 0$ there is a time
$T>T_0$ such that $\Phi_{T}(x) \in U$; the flow $\Phi$ itself is
\emph{recurrent} if almost every point (with respect to $\mu$) is
recurrent.

In our setting, the billiard flow $\phi^{\theta}$ is the flow at
constant unit speed in direction $\theta$, bouncing off at equal
angles upon hitting the rectangular obstacles.  Regular trajectories
are those which never hit a corner (where the billiard flow is
undefined).  In the whole paper, \emph{direction} is to be understood
as \emph{slope}.  We prove the following results:

\begin{theo}
\label{billiard}
If the rectangular obstacles have dimensions $(a,b)\in\EEE$, then, for
the billiard table $T_{a,b}$:
\begin{itemize}
\item
there is a subset $P$ of $\Q$, dense in $\R$, such that 
every regular trajectory starting with direction in $P$ is periodic;
\item
for almost every direction, the billiard flow is recurrent with
respect to the natural phase volume.
\end{itemize}
\end{theo}

Since $\EEE$ is countable, the set of directions of full measure in
the last statement can even be chosen independent of the parameter
$(a,b)\in\EEE$.

\medskip

For $j \in \N$ denote by $\log_j$ the $j$-th iterate of the logarithm
function, i.e., $\log_j = \log \circ \dots \circ \log$ ($j$ times).

\smallskip

Setting different parity conditions and defining
\begin{align*}
\EEE' & = 
\bigl\{\, (a,b) = (p/q,r/s) \in \Q \times \Q :
\\
& \phantom{={}}\quad
(p,q) = (r,s) =1,\quad 0<p<q,\quad 0<r<s,
\\[-3pt]
& \phantom{={}}\quad
p,r \textrm{ even,}\quad q,s \textrm{ odd}\,\bigr\},
\end{align*}
we have

\begin{theo}
\label{thm:billiard-other-parity}
If the rectangular obstacles have dimensions $(a,b)\in\EEE'$, then, for
the billiard table $T_{a,b}$:
\begin{itemize}
\item there is no direction $\alpha \in \Q$ such that all regular 
trajectories starting with direction $\alpha$ are periodic;
\item
there is a subset $P$ of $\Q$, dense in $\R$, such that 
no trajectory starting with direction in $P$ is periodic;
\item 
$\forall\,k \geq 1$, for a.e.\ $\theta$, $\displaystyle\limsup_{t \to \infty} \frac{\dist(\phi_{t}^{\theta}x,x)}
{\prod_{j=1}^k\log_{ j} t} = \infty$ almost surely.
\end{itemize}
\end{theo}

As a corollary, we obtain a result for a dense $G_{\delta}$ of
parameters in $\XXX$:

\begin{coro}
\label{G-delta}
Consider any closed $\YYY \subset \XXX$ for which $\EEE \cap \YYY$ and $\EEE'
\cap \YYY$ are dense in
$\YYY$.  Then there is a residual set $G \subset \YYY$ such that,
for each  $(a,b) \in G$,
\begin{itemize}
\item the billiard flow on
$T_{a,b}$ is recurrent;
\item
the set of periodic points is dense in the phase space of $T_{a,b}$;
\item 
$\forall\,k \geq 1$, for a.e.\ $\theta$, $\displaystyle\limsup_{t \to \infty} \frac{\dist(\phi_{t}^{\theta}x,x)}
{\prod_{j=1}^k\log_{ j} t} = \infty$ almost surely.
\end{itemize}
\end{coro}

The results of Theorems \ref{billiard} and
\ref{thm:billiard-other-parity} can be rephrased (and sharpened) in
terms of translation surfaces.  By a standard construction consisting
in unfolding the trajectories, the billiard flow in a given direction
is replaced by a linear flow on a translation surface $X_{a,b}$ (of
infinite area).  This translation surface of infinite area is a
$\Z^2$-covering of a finite-degree covering of a (finite-area) compact
translation surface $Y_{a,b}$ of genus $2$ with one singular point,
which is square-tiled when $a$ and $b$ are rational.

We use the same notation $\phi_t^\theta$ for the linear flow in 
direction $\theta$ on $X_{a,b}$ or $Y_{a,b}$ as for the billiard flow
in direction $\theta$ on $T_{a,b}$.
The flow $\phi_t^\theta$, or the direction $\theta$, is called
\emph{completely periodic} if every regular trajectory in direction
$\theta$ is closed; note that a given direction could be completely
periodic for $Y_{a,b}$ while not for $X_{a,b}$.
The flow $\phi_t^\theta$, or the direction $\theta$, is called
\emph{strongly parabolic} if in addition $X_{a,b}$ decomposes into
an infinite number of cylinders isometric to each other.

\begin{theo}
\label{main}
For rectangular obstacles with dimensions $(a,b)\in\EEE$:
\begin{itemize}
\item
the set $P \subset \Q$ of strongly parabolic 
directions for $X_{a,b}$ is dense in $\R$.
\item
in almost every direction, the linear flow on $X_{a,b}$ is
recurrent.
\end{itemize}
\end{theo}

\begin{remk}
\label{remk:a=b=1/2}
When $a = b = 1/2$, strongly parabolic directions are exactly the
rational directions $u/v$ such that $u$ and $v$ are both odd, and 
they are the only completely periodic directions.
\end{remk}

\begin{theo}
\label{remk:parity}
For rectangular obstacles with dimensions $(a,b)\in\EEE'$:
\begin{itemize}
\item
there are no completely periodic directions on $X_{a,b}$;
\item
there is a subset $P'$ of $\Q$, dense in $\R$, of directions in which
the linear flow on $Y_{a,b}$ is completely periodic with one cylinder,
and thus there are no periodic orbits on $X_{a,b}$ in these 
directions;
\item 
on $X_{a,b}$, $\forall\,k \geq 1$, for a.e.\ $\theta$, $\displaystyle \limsup_{t \to \infty}
\frac{\dist(\phi_{t}^{\theta}x,x)}{\prod_{j=1}^k\log_{ j} t} = \infty$
almost surely.
\end{itemize}
\end{theo}

\begin{remk}
\label{remk:prod-log-j}
In Theorem 2, Corollary 3 and Theorem 6, we suspect that the function
$\prod_{j = 1}^k \log_j$ could be replaced by a faster-growing
function, such as possibly $t\mapsto t^\lambda$ for a suitable
$\lambda$.
\end{remk}

\subsection{Reader's guide}

Our approach is based on results on compact translation surfaces.  The
fundamental paper of Veech \cite{Ve} was the inspiration for many
papers about Veech surfaces and (compact) square-tiled surfaces, for
which the situation is particularly well understood in genus two.

If $(a, b) \in \EEE$, we remark that some `good' one-cylinder
directions on $Y_{a,b}$ are strongly parabolic directions for
$X_{a,b}$.  This is our main statement.  This relies on a careful
study of the Weierstrass points on $Y_{a,b}$ developed by McMullen
(see \cite{Mc}).  From there, good Diophantine approximation of almost
every irrational number by these periodic directions (due to the fact
that the set of good one-cylinder directions contains the orbit of a
cusp of a Fuchsian group of finite covolume) implies the recurrence of
the flow for almost every direction when $(a,b) \in \EEE$.  Similar
arguments hold when $(a,b)\in\EEE'$.  Using a standard trick in
billiard theory, we get Corollary \ref{G-delta}.

\subsection{Acknowledgments}

We thank C.~McMullen for showing us how to give a good definition of
the set $\EEE$.  We thank F.~Paulin who explained to us the part
concerning diophantine approximation.  We also thank F.~Valdez for
helpful discussions concerning the definition of a non compact
translation surface.

The first and second authors are partially supported by Projet blanc
ANR: ANR-06-BLAN-0038.

\section{Background}

\subsection{Translation surfaces} 

In this section, we briefly introduce the basic notions of
Teichmüller dynamics.  For more on translation surfaces, see for
instance \cite{MaTa}, \cite{Vi}, \cite{Zo}.

A surface is called a \emph{translation surface} if it can be obtained
by edge-to-edge gluing of polygons in the plane, only using translations
(the polygons need not be compact or finitely many but should be
at most countably many, and the lengths of sides and diagonals should
be bounded away from zero).  The translation structure induces a flat
metric with conical singularities.  Note that if the gluings produce
an infinite angle at some vertex, no neighborhood of this vertex is
homeomorphic to an open set in the plane, so that with this point the
resulting object cannot properly be called a surface; however this
does not occur in this paper. A discussion about singularities on
 non compact translation surfaces can be found in 
\cite{Bo}.

A \emph{saddle connection} is a geodesic segment for the flat metric
starting and ending at a singularity, and containing no singularity in
its interior.

A \emph{cylinder} on a translation surface is a maximal connected union
of homotopic simple closed geodesics.  If the surface is compact and
the genus is greater than one then every cylinder is bounded by saddle
connections.  A cylinder has a length (or circumference) $c$ and a
height $h$.  The \emph{modulus} of a cylinder is $\mu=h/c$.

On a noncompact translation surface, a \emph{strip} is a maximal
connected union of parallel biinfinite geodesics.

On a compact translation surface, a direction $\theta$ is called {\em
periodic} if the translation surface is the union of the closures of
cylinders in this direction, and \emph{parabolic} if moreover the
moduli of all the cylinders are commensurable.

On a translation surface, a direction $\theta$ is \emph{strongly
parabolic} if the surface is the union of closed cylinders in this
direction and the cylinders are isometric to each other.

\subsection{Square-tiled surfaces}
\label{section-squaretiled}

A \emph{square-tiled surface} is a translation surface obtained by
edge-to-edge gluing of unit squares.  Consequently it is a covering of
the torus $\R^2/\Z^2$ ramified only over the origin.  A square-tiled
surface is \emph{primitive} if the lattice generated by the vectors of
saddle connections is equal to $\Z^2$.  Abusing notations, we will
call square-tiled a translation surface obtained by gluing copies of a
fixed rectangle instead of squares.  We recall that on a compact
square-tiled surface, a direction has rational slope if and only if it
is periodic if and only if it is parabolic.  A theorem by Veech
implies that on a compact square-tiled surface, every irrational
direction is minimal and uniquely ergodic \cite{Ve}.

\subsection{Genus 2 surfaces, L-shaped surfaces}
\label{section-genustwo}

For finite area translation surfaces, the angles around the
singularities are multiples of $2\pi$.  The family of these integers
is the combinatorics of the surface.  The moduli space of translation
surfaces with fixed combinatorics is called a \emph{stratum}.  In genus
2, there are 2 strata.  One of them is made of surfaces with one
singularity of angle $6\pi$; it is called $\HHH(2)$.  The other one is
made of surfaces with two singularities of angle $4\pi$; it is called
$\HHH(1,1)$.  Every genus 2 translation surface is hyperelliptic.

An L-shaped surface $\Lsurf(\alpha, \beta, \gamma, \delta)$ is a
translation surface defined by 4 parameters: the lengths of the
horizontal saddle connections $\alpha$, $\gamma$ and the lengths of
the vertical saddle connections $\beta$, $\delta$, see figure
\ref{fig:L}, where segments with same labels are glued together.  An
L-shaped surface belongs to the stratum $\HHH(2)$.  Its six
Weierstrass points are depicted on figure \ref{fig:Wpts}.  When the
parameters $\alpha, \beta, \gamma, \delta$ are rational, it is a
square-tiled surface.

\begin{figure}[ht]
\begin{minipage}[ht]{0.48\linewidth}
\centering
\begin{tikzpicture}[scale=0.9]
  \draw [very thick] (0,0)
    -- node [below] {$\gamma$} ++(3,0)
    -- node [below] {$\alpha$} ++(2.5,0)
    -- node [right] {$\delta$} ++(0,1.75)
    -- node [above] {$\alpha$} ++(-2.5,0)
    -- node [right] {$\beta$} ++(0,1.25)
    -- node [above] {$\gamma$} ++(-3,0)
    -- node [left] {$\beta$} ++(0,-1.25)
    -- node [left] {$\delta$} (0,0)
    -- cycle;
  \draw[fill] (0,0) circle (2pt);
  \draw[fill] (3,0) circle (2pt);
  \draw[fill] (5.5,0) circle (2pt);
  \draw[fill] (5.5,1.75) circle (2pt);
  \draw[fill] (3,1.75) circle (2pt);
  \draw[fill] (3,3) circle (2pt);
  \draw[fill] (0,3) circle (2pt);
  \draw[fill] (0,1.75) circle (2pt);
\end{tikzpicture}
\caption{\mbox{L-shaped surface}}
\label{fig:L}
\end{minipage}
\begin{minipage}[ht]{0.48\linewidth}
\centering
\begin{tikzpicture}[scale=0.9]
  \draw [very thick] (0,0)
    -- ++(3,0)
    -- node [below] {\phantom{$\gamma$}} ++(2.5,0)
    -- ++(0,1.75)
    -- ++(-2.5,0)
    -- ++(0,1.25)
    -- ++(-3,0)
    -- ++(0,-1.25)
    -- cycle;
  \draw[fill] (0,0) circle (2pt) node [above right] {$D$};
  \draw[fill] (3,0) circle (2pt) node [above right] {$D$};
  \draw[fill] (5.5,0) circle (2pt) node [above right] {$D$};
  \draw[fill] (5.5,1.75) circle (2pt) node [above right] {$D$};
  \draw[fill] (3,1.75) circle (2pt) node [above right] {$D$};
  \draw[fill] (3,3) circle (2pt) node [above right] {$D$};
  \draw[fill] (0,3) circle (2pt) node [above right] {$D$};
  \draw[fill] (0,1.75) circle (2pt) node [above right] {$D$};
  \draw (4.25,0) circle (2.5pt) node [above right] {$E$};
  \draw (4.25,0.875) circle (2.5pt) node [above right] {$B$};
  \draw (4.25,1.75) circle (2.5pt) node [above right] {$E$};
  \draw (3,2.375) circle (2.5pt) node [above right] {$F$};
  \draw (1.5,2.375) circle (2.5pt) node [above right] {$C$};
  \draw (0,2.375) circle (2.5pt) node [above right] {$F$};
  \draw (1.5,0.875) circle (2.5pt) node [above right] {$A$};
\end{tikzpicture}
\caption{\mbox{Weierstrass points}}
\label{fig:Wpts}
\end{minipage}
\end{figure}

The position of Weierstrass points with respect to the squares that
tile the surface may be among the following: corners, centers,
midpoints of horizontal or vertical edges.  When a Weierstrass point
is located at a corner, we call it an \emph{integer Weierstrass
point}.

The action of $\SL_{2}(\Z)$ (see the following subsection) on
primitive square-tiled surfaces in $\HHH(2)$ preserves the number of
integer Weierstrass points.  This provides an invariant distinguishing
between $\SL_{2}(\Z)$ orbits of primitive square tiled surfaces in
$\HHH(2)$ (see \cite{HuLe}, \cite{Mc}).  If the number of squares is
odd and at least $5$, there are two orbits: in one orbit (orbit $A$)
the surfaces have 1 integer Weierstrass point (the singularity); in
the other one (orbit $B$), there are 3 integer Weierstrass points.
In the proof of Lemma \ref{1-cyl}, we will use an a priori finer
invariant introduced by Kani (see \cite{Ka1}, \cite{Ka2}).

\subsection{Veech groups}

Given any translation surface $(X, \omega)$, an \emph{affine
diffeomorphism} is an orientation preserving homeomorphism of $X$ that
permutes the singularities of the flat metric and acts affinely on the
polygons defining $X$.  The group of affine diffeomorphisms is denoted
by $\Aff(X,\omega)$.  The image of the derivation map
$$d : 
   \left\{
   \begin{matrix}
    \Aff(X,\omega) \to \GL_{2}(\R) \\
    f \mapsto df
    \end{matrix}
   \right.
$$
is called the \emph{Veech group}, and denoted by $\SL(X, \omega)$.  If
$(X,\omega)$ is a compact translation surface, then $\SL(X, \omega)$
is a Fuchsian group (discrete subgroup of $\SL_{2}(\R)$).  If
$(X,\omega)$ is a primitive square-tiled surface then its Veech group
is a subgroup of $\SL_{2}(\Z)$, for the following reason.  First of
all, an affine diffeomorphism sends a saddle connection to a saddle
connection, thus it sends the standard basis of $\R^2$ to vectors with
integer coordinates.  Thus, elements of $\SL(X, \omega)$ have integral
entries.  By the same reasoning, the inverse of such element has also
integer entries and thus $\SL(X, \omega)$ is a subgroup of
$\SL_{2}(\Z)$.

For surfaces in $\HHH(2)$, the affine group is isomorphic to the 
Veech group (see for instance proposition 4.4 in \cite{HuLe}).

\subsection{Construction of $X_{a,b}$}

Label each rectangular obstacle by the element of $\Z^2$ giving its
position.  Its four sides are $\ell$ (left), $r$ (right), $b$
(bottom), $t$ (top).  We recall the classical construction of a
translation surface from a polygonal billiard.  The idea is to unfold
the trajectories and to reflect the table.  To get the surface
$X_{a,b}$, we need four copies of the billiard table $T_{a,b}$ because
the angles of the table are multiples of $\pi/2$.  We label the copies
by $I$, $II$, $III$, $IV$.

\begin{figure}[ht]
\pgfdeclaresnake{diamondline}{initial}
{
\state{initial}[width=10pt]
{
\pgfpathlineto{\pgfpoint{5pt}{0pt}}
\pgfpathlineto{\pgfpoint{4pt}{0.75pt}}
\pgfpathlineto{\pgfpoint{3pt}{0pt}}
\pgfpathlineto{\pgfpoint{4pt}{-0.75pt}}
\pgfpathlineto{\pgfpoint{5pt}{0pt}}
}
\state{final}
{
\pgfpathlineto{\pgfpoint{\pgfsnakeremainingdistance}{0pt}}
}
}

\begin{tikzpicture}
  \tikzstyle{a line}=[very thick, dash pattern=on 8pt off 2pt];
  \tikzstyle{b line}=[very thick, dash pattern=on 5pt off 8pt];
  \tikzstyle{c line}=[very thick,
%     snake=triangles,segment object length=10pt,segment length=7pt,
%     segment amplitude=.22mm,
%     line around snake=1.55mm];
    snake=diamondline,line before snake=3.5mm];
  \tikzstyle{d line}=[very thick,
    dash pattern=on 5pt off 2pt on 1pt off 2pt on 1pt off 2pt];
  \tikzstyle{e line}=[line width=1pt,
    dash pattern=on 3pt off 2pt];
  \tikzstyle{f line}=[very thick,
%     dash pattern=on 3pt off 2pt on 3pt off 2pt on 1pt off 2 pt
    dash pattern=on 8pt off 2pt on 1pt off 2pt];
  \tikzstyle{g line}=[very thick,snake=zigzag,
    segment amplitude=.3mm,segment length=2.5mm,line around snake=1mm];
  \tikzstyle{h line}=[very thick,snake=saw,
    segment amplitude=.4mm,segment length=2mm,line around snake=1mm];
  \draw (-5.5,0) -- (5.5,0);
  \draw (0,-5.5) -- (0,5.5);
  \node at (-0.5,-0.5) {$I$};
  \node at (0.5,-0.5) {$II$};
  \node at (-0.5,0.5) {$III$};
  \node at (0.5,0.5) {$IV$};
  % I
  \begin{scope}[scale=2.5,xshift=-1.15cm,yshift=-1.15cm]
    \clip (-0.9,-0.9) rectangle +(1.8,1.8);
    \foreach \z in {(-1,-1),(-1,0),(-1,1),(0,-1),(0,1),(1,-1),(1,0),(1,1)}
      \draw[very thick] \z +(-0.3,-0.3) rectangle +(0.3,0.3);
    \node at (0,0) {$\vct{m}{n}$};
    \draw[a line] (-0.3,-0.3) -- +(0.6,0);
    \draw[b line] (0.3,-0.3) -- +(0,0.6);
    \draw[c line] (-0.3,0.3) -- +(0.6,0);
    \draw[d line] (-0.3,-0.3) -- +(0,0.6);
    \node at (-0.4,0) {$\ell$};
    \node at (0,-0.4) {$b$};
    \node at (0.4,0) {$r$};
    \node at (0,0.4) {$t$};
  \end{scope}
  % II
  \begin{scope}[scale=2.5,xshift=1.15cm,yshift=-1.15cm]
    \clip (-0.9,-0.9) rectangle +(1.8,1.8);
    \foreach \z in {(-1,-1),(-1,0),(-1,1),(0,-1),(0,1),(1,-1),(1,0),(1,1)}
      \draw[very thick] \z +(-0.3,-0.3) rectangle +(0.3,0.3);
    \node at (0,0) {$\vct{-m}{n}$};
    \draw[g line] (-0.3,-0.3) -- +(0.6,0);
    \draw[d line] (0.3,-0.3) -- +(0,0.6);
    \draw[h line] (-0.3,0.3) -- +(0.6,0);
    \draw[b line] (-0.3,-0.3) -- +(0,0.6);
    \node at (-0.4,0) {$\ell$};
    \node at (0,-0.4) {$b$};
    \node at (0.4,0) {$r$};
    \node at (0,0.4) {$t$};
  \end{scope}
  % III
  \begin{scope}[scale=2.5,xshift=-1.15cm,yshift=1.15cm]
    \clip (-0.9,-0.9) rectangle +(1.8,1.8);
    \foreach \z in {(-1,-1),(-1,0),(-1,1),(0,-1),(0,1),(1,-1),(1,0),(1,1)}
      \draw[very thick] \z +(-0.3,-0.3) rectangle +(0.3,0.3);
    \node at (0,0) {$\vct{m}{-n}$};
    \draw[c line] (-0.3,-0.3) -- +(0.6,0);
    \draw[e line] (0.3,-0.3) -- +(0,0.6);
    \draw[a line] (-0.3,0.3) -- +(0.6,0);
    \draw[f line] (-0.3,-0.3) -- +(0,0.6);
    \node at (-0.4,0) {$\ell$};
    \node at (0,-0.4) {$b$};
    \node at (0.4,0) {$r$};
    \node at (0,0.4) {$t$};
  \end{scope}
  % IV
  \begin{scope}[scale=2.5,xshift=1.15cm,yshift=1.15cm]
    \clip (-0.9,-0.9) rectangle +(1.8,1.8);
    \foreach \z in {(-1,-1),(-1,0),(-1,1),(0,-1),(0,1),(1,-1),(1,0),(1,1)}
      \draw[very thick] \z +(-0.3,-0.3) rectangle +(0.3,0.3);
    \node at (0,0) {$\vct{-m}{-n}$};
    \draw[h line] (-0.3,-0.3) -- +(0.6,0);
    \draw[f line] (0.3,-0.3) -- +(0,0.6);
    \draw[g line] (-0.3,0.3) -- +(0.6,0);
    \draw[e line] (-0.3,-0.3) -- +(0,0.6);
    \node at (-0.4,0) {$\ell$};
    \node at (0,-0.4) {$b$};
    \node at (0.4,0) {$r$};
    \node at (0,0.4) {$t$};
  \end{scope}
\end{tikzpicture}
\caption{Side-pairings for $X_{a,b}$}
\label{fig:Xab-side-pairings}
\end{figure}
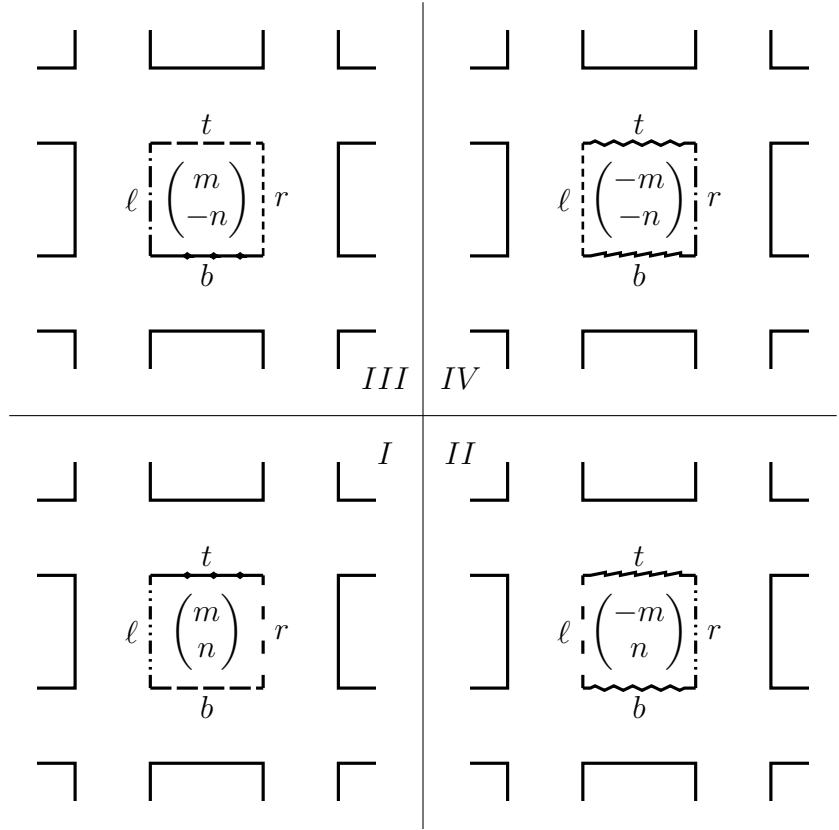

The following sides are glued together (see figure 
\ref{fig:Xab-side-pairings}), for all couples $(m,n)\in\Z^2$:
\begin{itemize}
\item
$(\ell, (m,n), I)$ is identified with $(r, (-m,n), II)$
\item
$(r, (m,n), I)$ is identified with $(\ell, (-m,n), II)$
\item
$(b, (m,n), I)$ is identified with $(t, (m,-n), III)$
\item
$(t, (m,n), I)$ is identified with $(b, (m,-n), III)$
\item
$(\ell, (m,n), III)$ is identified with $(r, (-m,n), IV)$
\item
$(r, (m,n), III)$ is identified with $(\ell, (-m,n), IV)$
\item
$(b, (m,n), II)$ is identified with $(t, (m,-n), IV)$
\item
$(t, (m,n), II)$ is identified with $(b, (m,-n), IV)$
\end{itemize}
The surface $X_{a,b}$ obtained this way has no boundary, infinite area
and infinite genus.

% ------------------------------------------------------------------
 
\section{Proof of Theorem \ref{main}}

We prove Theorem \ref{main} and note that Theorem \ref{billiard} follows 
immediately.

\subsection{Construction of a finite area square-tiled surface.}

Given $a, b \in (0,1)$, we construct the following (finite area)
translation surface.  A fundamental domain is $[0,1]\times [0,1]
\setminus [(1-a)/2,(1+a)/2]\times [(1-b)/2,(1+b)/2]$.  Parallel sides
of the same length are identified (see figure
\ref{fig:finite-area-surface}).  We denote this surface by $Y_{a,b}$.
$Y_{a,b}$ is a surface of genus two with one conical point with angle
$6\pi$.  It belongs to the stratum $\HHH(2)$.  $Y_{a,b}$ is a quotient
of $X_{a,b}$.  We denote by $\pi$ the projection from $X_{a,b}$ to
$Y_{a,b}$.  We are going to derive properties of the linear flows on
$X_{a,b}$ from results on square-tiled surfaces in $\HHH(2)$.  The
surface $Y_{a,b}$ is an L-shaped surface of parameters
$(1-a,a,1-b,b)$.  The 6 Weierstrass points are $A$, $B$, $C$, $D$,
$E$, $F$ on figure \ref{fig:Wpts}.  When $a$ and $b$ are rational, it
is a square-tiled surface (in fact, it is tiled by small rectangles).

% \begin{figure}[ht]
% \begin{tikzpicture}
%   \draw [very thick] (-1.5,-0.875)
%     -- node [below] {$\alpha+\gamma$} ++(5.5,0)
%     -- node [right] {$\beta+\delta$} ++(0,3)
%     -- node [above] {$\alpha+\gamma$} ++(-5.5,0)
%     -- node [left] {$\beta+\delta$} ++(0,-3)
%     -- cycle;
%   \draw [very thick] (0,0)
%     -- node [above] {$\alpha$} ++(2.5,0)
%     -- node [left] {$\beta$} ++(0,1.25)
%     -- node [below] {$\alpha$} ++(-2.5,0)
%     -- node [right] {$\beta$} ++(0,-1.25)
%     -- cycle;
%   \draw[fill] (0,0) circle (2pt);
%   \draw[fill] (2.5,0) circle (2pt);
%   \draw[fill] (2.5,1.25) circle (2pt);
%   \draw[fill] (0,1.25) circle (2pt);
% \end{tikzpicture}
% \caption{Finite area surface}
% \label{fig:finite-area-surface}
% \end{figure}

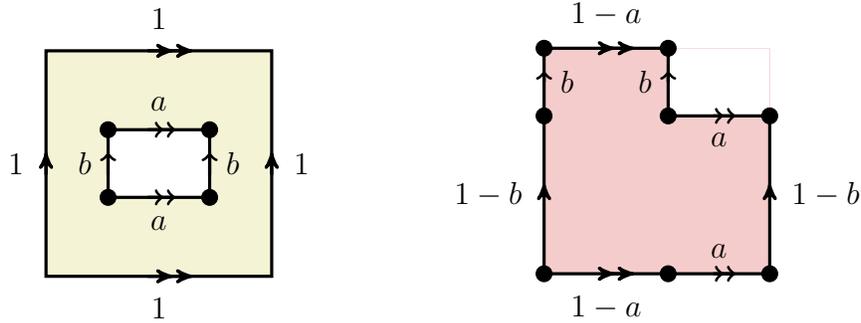
\begin{figure}[ht]
\begin{tikzpicture}[scale=1.5]
  \filldraw[fill=yellow!80!black!20,even odd rule]
    (-1,-1) rectangle (1,1) (-0.45,-0.3) rectangle (0.45,0.3);
  \draw [very thick] (-1,-1) rectangle (1,1);
  \draw [very thick] (-0.45,-0.3) rectangle (0.45,0.3);
  \draw[fill] (-0.45,-0.3) circle (2pt);
  \draw[fill] (-0.45,0.3) circle (2pt);
  \draw[fill] (0.45,-0.3) circle (2pt);
  \draw[fill] (0.45,0.3) circle (2pt);
  \draw [->>,very thick] (-0.1,-0.3) -- (0.15,-0.3);
  \draw [->>,very thick] (-0.1,0.3) -- (0.15,0.3);
  \draw [->,very thick] (-0.45,-0.1) -- (-0.45,0.1);
  \draw [->,very thick] (0.45,-0.1) -- (0.45,0.1);
  \begin{scope}[>=stealth']
  \draw [->>,very thick] (-0.1,-1) -- (0.3,-1);
  \draw [->>,very thick] (-0.1,1) -- (0.3,1);
  \draw [->,very thick] (-1,-0.1) -- (-1,0.1);
  \draw [->,very thick] (1,-0.1) -- (1,0.1);
  \node at (1.1,0) [right] {$1$};
  \node at (-1.1,0) [left] {$1$};
  \node at (0,1.1) [above] {$1$};
  \node at (0,-1.1) [below] {$1$};
  \node at (0.5,0) [right] {$b$};
  \node at (-0.5,0) [left] {$b$};
  \node at (0,0.37) [above] {$a$};
  \node at (0,-0.37) [below] {$a$};
  \end{scope}
\end{tikzpicture}%
\qquad\qquad
\begin{tikzpicture}[scale=1.5]
  \fill[fill=red!80!black!20,even odd rule]
    (-1,-1) rectangle (1,1) (0.1,0.4) rectangle (1,1);
  \draw [very thick] (-1,-1) -- (1,-1) -- (1,0.4) --
    (0.1,0.4) -- (0.1,1) -- (-1,1) -- cycle;
  \draw[fill] (0.1,0.4) circle (2pt);
  \draw[fill] (0.1,1) circle (2pt);
  \draw[fill] (0.1,-1) circle (2pt);
  \draw[fill] (1,0.4) circle (2pt);
  \draw[fill] (-1,0.4) circle (2pt);
  \draw[fill] (-1,-1) circle (2pt);
  \draw[fill] (-1,1) circle (2pt);
  \draw[fill] (1,-1) circle (2pt);
  \begin{scope}
  \draw [->>,very thick] (0.5,-1) -- (0.7,-1);
  \draw [->>,very thick] (0.5,0.4) -- (0.7,0.4);
  \draw [->,very thick] (-1,0.7) -- (-1,0.8);
  \draw [->,very thick] (0.1,0.7) -- (0.1,0.8);
  \end{scope}
  \begin{scope}[>=stealth']
  \draw [->>,very thick] (-0.4,-1) -- (-0.2,-1);
  \draw [->>,very thick] (-0.4,1) -- (-0.2,1);
  \draw [->,very thick] (-1,-0.3) -- (-1,-0.2);
  \draw [->,very thick] (1,-0.3) -- (1,-0.2);
  \end{scope}
  \node at (1.1,-0.3) [right] {$1-b$};
  \node at (-1.1,-0.3) [left] {$1-b$};
  \node at (-0.45,1.1) [above] {$1-a$};
  \node at (-0.45,-1.1) [below] {$1-a$};
  \node at (0.55,0.35) [below] {$a$};
  \node at (0.55,-0.95) [above] {$a$};
  \node at (-0.95,0.7) [right] {$b$};
  \node at (0.05,0.7) [left] {$b$};
\end{tikzpicture}%
\caption{Finite area surface $Y_{a,b}$: two views. Either cut
out a rectangle of dimensions $a\times b$ from a unit square and glue
parallel and same length sides of the resulting polygon together, or,
equivalently, glue facing sides of an L-shaped polygon (which amounts
to cutting out the $a\times b$ rectangle at the top right corner of
the unit square).}
\label{fig:finite-area-surface}
\end{figure}

\subsection{Good one-cylinder directions on square-tiled surfaces in
$\boldsymbol{\HHH(2)}$}

We recall that a translation surface in $\HHH(2)$ is hyperelliptic.

McMullen \cite{Mc} proved (see Hubert-Lelièvre \cite{HuLe} for
earlier results) that every square-tiled surface in $\HHH(2)$ contains
a one-cylinder direction.  Here, we need more, and we will use the 
following lemma.

\begin{lemma} \label{1-cyl}
Assume that $a = p/q$, $b = r/s$ are rational and belong to $\EEE$.
Then the surface $Y_{a,b}$ contains a dense set of one-cylinder
directions with $E$ and $F$ on the waist curve of the cylinder.
\end{lemma}

\begin{defi}
In the sequel, we will call \emph{good one-cylinder direction} a
one-cylinder direction with $E$ and $F$ on the waist curve of
the cylinder.
\end{defi}

\begin{proof}
We first prove that there exists one good one-cylinder direction.  We
rescale the surface, horizontally by $q$ and vertically by $s$, to
make it a primitive square-tiled surface.  The surface $Y_{a,b}$
becomes the L-shaped surface $\Lsurf(p,q-p,r,s-r)$.  As $p, r, q-p,
s-r$ are odd, this surface belongs to the orbit A (defined in
subsection \ref{section-genustwo}).

Moreover, as C.~McMullen pointed out to us, Weierstrass points of a
surface in orbit A have the following property: 3 regular Weierstrass
points project to the same point of the torus, the singularity
projects to the origin of the torus, the two remaining Weierstrass
points project to the two other Weierstrass points of the torus.  This
gives a way to distinguish intrinsically these last 2 Weierstrass
points.  This property is $\SL_{2}(\Z)$ invariant.  It is known as
Kani's invariant.  McMullen (see \cite{Mc} Appendix A) proves that,
for a one-cylinder decomposition, the points on the waist curve of the
cylinder are the 2 distinguished Weierstrass points.  Therefore, it is
enough for us to check that the regular Weierstrass points $A$, $B$,
$C$ project to the same point on the torus, which is true, since, $p$,
$r$, $q-p$, and $s-r$ being odd, these points project to $(1/2,1/2)$
on the torus.  Consequently $E$ and $F$ belong to the waist curve of
\emph{every} one-cylinder decomposition.

Now, we prove the density.  The image of a one-cylinder direction
under an affine map is a one-cylinder direction.  Moreover the orbit
of any point of $P^1(\R)$ under a lattice in $\SL_{2}(\R)$ is dense in
$P^1(\R)$.  Therefore the set of good one-cylinder direction is dense
in the circle.
\end{proof}

\begin{remk}
Using the same method, we can prove that if the parameters $(a,b)$
belong to $\Q\times\Q \setminus \EEE$, then $Y_{a,b}$ has no good
one-cylinder direction.
\end{remk}

\subsection{Periodic orbits on $X_{a,b}$} \label{periodic orbits}

\begin{lemma} \label{periodic-orbits}
Good one-cylinder directions $Y_{a,b}$ are strongly parabolic
for $X_{a,b}$.  In such a direction, the length of a periodic orbit on
$X_{a,b}$ is twice the length of its projection on $Y_{a,b}$.
\end{lemma}

\begin{remk}
The proof has a fair amount in common with the proof of a 
similar result for $\Z$-covers of compact translation surfaces which 
appears in Hubert-Schmithüsen \cite{HuSc}.
\end{remk}

\begin{proof}
Let $\theta$ be a good one-cylinder direction on $Y_{a,b}$.  In
$Y_{a,b}$, every non singular orbit of direction $\theta$ is a
translate of the orbit of $E$.  Thus any non singular trajectory on
$X_{a,b}$ in direction $\theta$ is a translate of the orbit of some
preimage of $E$ (they belong to the same strip).  Thus it is enough to
prove that the orbit of any preimage of $E$ on $X_{a,b}$ is closed and
to compute its length (in fact, since the table $X_{a,b}$ is periodic,
the proof for one such orbit is enough).

We use the billiard flow for this proof instead of the linear flow
because it seems to be more simple.

\begin{claim}
\label{symmetry}
Let $\alpha$ be any direction other than horizontal or vertical,
$\phi$ the flow in direction $\alpha$, $\hat{E}$ a preimage of $E$
in $T_{a,b}$, $\hat{F}$ a preimage of $F$ in $T_{a,b}$.  Then, for 
each $t$, $\phi_{t}(\hat{E})$ and $\phi_{-t}(\hat{E})$ are symmetric with
respect to the vertical line containing $\hat{E}$, and
$\phi_{t}(\hat{F})$ and $\phi_{-t}(\hat{F})$ are symmetric with
respect to the horizontal line containing $\hat{F}$.
\end{claim}

(The figures in the appendix may be of help in reading this proof.)

As $\hat{E}$ belongs to a horizontal side, before the first reflection
on the boundary, the result is certainly true.  Now, as the table is
symmetric with respect to the vertical line containing $\hat{E}$, the
trajectories in positive and negative time remain symmetric forever.
The proof is similar for $\hat{F}$, so the claim is proved.

\begin{remk}
Consider a good one-cylinder direction for $Y_{a,b}$; recall that this
means $Y_{a,b}$ decomposes into a single cylinder, whose waist curve
contains $E$ and $F$.  If $\ell$ is the length of the cylinder, after
flowing during time $\ell/2$ from $E$ we reach $F$.
\end{remk}

We have all the ingredients to prove Lemma \ref{periodic-orbits}.
Call $(x,y)$ the coordinates of $\hat{E}$.  From the previous remark,
$\phi_{\ell/2}(\hat{E})$ is a preimage $\hat{F}$ of $F$ with
coordinates $(u,v)$.  By claim \ref{symmetry},
$\phi_{\ell}(\hat{E}) = (x, 2v -y)$.  Since $\phi_{\ell}(\hat{E}) =
\phi_{\ell/2}(\hat{F})$, using once more the remark,
$\phi_{\ell}(\hat{E})$ is a preimage $\tilde{E}$ of $E$.  By
claim \ref{symmetry} again, $\phi_{3\ell/2}(\hat{E}) = (2x -u,v)$.
Continuing the same way, $\phi_{3\ell/2}(\hat{E})$ is $\tilde{F}$ a
preimage of $F$.  Applying the symmetry argument,
$$\phi_{2\ell}(\hat{E}) = (x, 2v - (2v -y)) = (x,y) = \hat{E}.$$

We have to check that the outgoing vectors at $\hat{E}$ at times $0$ and
$2\ell$ have the same direction.  Let $\varepsilon$ be small enough.
Once more, using the symmetry argument for $\tilde{E}$, we see that
$\phi_{\varepsilon}(\hat{E})$ and $\phi_{2 \ell -
\varepsilon}(\hat{E})$ are symmetric with respect to the vertical line
containing $\hat{E}$ and $\tilde{E}$.  As $\hat{E}$ is on a horizontal
boundary, $\phi_{2 \ell + \varepsilon}(\hat{E})$ and $\phi_{2 \ell -
\varepsilon}(\hat{E})$ are symmetric with respect to this line.  Thus
$\phi_{2 \ell + \varepsilon}(\hat{E}) = \phi_{\varepsilon}(\hat{E})$.
This exactly means that the trajectory is closed.

We then remark that the length cannot be less that $2 \ell$.  It is a
multiple of $\ell$ since it is $\ell$ on $Y_{a,b}$.  We already note
that $\hat{E}$ and $\tilde{E}$ are symmetric with respect to the
vertical line containing $\hat{F}$.  As the $y$-coordinate of any
preimage of $F$ differs from the $y$ coordinate of $\hat{E}$, then
$\hat{E}$ and $\tilde{E}$ are different.  Therefore the trajectory of
$\hat{E}$ has length $2\ell$ and the proof of Lemma
\ref{periodic-orbits} is complete.
\end{proof}

Lemmas \ref{1-cyl} and \ref{periodic-orbits} prove the first point of
Theorem \ref{main}.
We shall now prove the remark stated after Theorem \ref{main}, 
and then its second point.

\subsection{Proof Remark \ref{remk:a=b=1/2}}

\begin{lemma}
\label{lem:ABC-strip}
Let $(a,b)\in(0,1)^2$, and label by $A$, $B$, $C$, $D$, $E$, $F$ the
Weierstrass points of $Y_{a,b}$ as on figure \ref{fig:Wpts}.  Let
$\gamma$ be a regular trajectory on $Y_{a,b}$ containing two of the
Weierstrass points $A$, $B$, $C$.  Then any lift to $X_{a,b}$ of
$\gamma$ has infinite length (it is not a closed trajectory on 
$X_{a,b}$).
\end{lemma}

\begin{proof}[Proof of Lemma]
Let $\gamma$ be a closed geodesic on $Y_{a,b}$ containing $A$ and 
$B$, and let $\widehat{\gamma}$ be a corresponding billiard 
trajectory on $T_{a,b}$ (i.e. $\widehat{\gamma}$ unfolds to 
a lift of $\gamma$ on $X_{a,b}$).

Call $A_0$ and $B_0$ two points on $\widehat{\gamma}$ which correspond
to $A$ and $B$.  Since the billiard trajectory goes through $A_0$ and
$B_0$ in straight line (there is no reflection at these points) and
since these points are centers of symmetry for $T_{a,b}$, 
$\widehat{\gamma}$ also has $A_0$ and $B_0$ as centers of symmetry; 
therefore it is infinite.
\end{proof}

Recall that the Veech group of the surface $Y_{1/2,1/2}$ is generated
by $\mat{1}{2}{0}{1}$ and $\mat{0}{-1}{1}{0}$.  One-cylinder
directions have slope $n/m$ with $n,m$ odd, the other rational
directions correspond to periodic directions with 2 cylinders.  As
$(1/2,1/2)$ belong to $\EEE$, one-cylinder directions lift to periodic
directions on $X_{1/2,1/2}$.  To prove Remark \ref{remk:a=b=1/2}, it
now suffices to prove that in any two-cylinder direction for
$Y_{1/2,1/2}$, there is a regular geodesic on $X_{1/2,1/2}$ connecting
a lift of $A$ and a lift of $B$.

One more lemma helps with this.

\begin{lemma}
\label{lem:Wpts-orbits}
The orbits of the Weierstrass points of $Y_{1/2,1/2}$ under its 
affine group are $\{A, B, C\}$, $\{D\}$, $\{E, F\}$.
\end{lemma}

The proof is immediate, by looking at the action of $\mat{1}{2}{0}{1}$
and $\mat{0}{-1}{1}{0}$ which generate the Veech group.  (The affine 
group is isomorphic to the Veech group for surfaces in $\HHH(2)$.)

\bigskip

Now let $\theta = u/v$ be a two-cylinder direction for $Y_{1/2,1/2}$;
then $(v,u)$ is the image of $(1,0)$ under some element of the
Veech group $\SL(X_{1/2,1/2})$.  Let $f$ be the corresponding element
in $\Aff(Y_{1/2,1/2})$.

Since in the horizontal direction there is a geodesic on 
$Y_{1/2,1/2}$ connecting $A$ and $B$, in direction $\theta$
there is a geodesic on $Y_{1/2,1/2}$ connecting $f(A)$ and $f(B)$, 
which are two points among $A$, $B$, $C$, by Lemma \ref{lem:Wpts-orbits}.
We conclude by Lemma 
\ref{lem:ABC-strip}.

\subsection{Recurrence}

%\begin{defi}
%Let  $\Phi$ be a flow acting on a space $\omega$, a point $x \in \Omega$
%is recurrent if for every neighborhood $U$ of $x$ there is a time $T$ such
%that $\Phi_{T}(x) \in U$.
%
%If the space $\Omega$ is endowed with a measure $\mu$, the flow $\Phi$ is
%recurrent if almost every point (with respect to $\mu$) is recurrent.
%\end{defi}

\begin{proofof}{Theorem \ref{main}}
We prove the second statement of Theorem \ref{main}: if $(a,b) \in
\EEE$, the linear flow is recurrent on $X_{a,b}$ for almost every
irrational direction.  Without loss of generality, we assume that the
slope is less than 1.

We first note that, for every $\kappa>0$ there exists a subset
$\Theta_{\kappa}$ of $[0,1]$ of full measure satisfying: if $\theta$
belongs to $\Theta_{\kappa}$, there exists an infinite sequence of
rationals $(p_{n}/ q_{n}) \in \PPP$ such that
$$\lvert \theta - \frac{p_{n}}{q_{n}} \rvert \leq \frac{\kappa}{q_{n}^2}.$$
This is true for every cusp of a finite index subgroup of
$\text{SL}_{2}(\Z)$ (see \cite{Su}).  This can be rephrased in the following way.
Given a sequence $\e_{n} \to 0$, there exists a set of full measure
$\Theta$ such that if $\theta$ belongs to $\Theta$, there exists an infinite
sequence of rationals $(p_{n}/ q_{n}) \in \PPP$ such that
\begin{equation}
\lvert \theta - \frac{p_{n}}{q_{n}} \rvert \leq \frac{\e_{n}}{q_{n}^2}.
\label{eq:dioph}
\end{equation}
As we proved that $\PPP$ contains all good one-cylinder
directions for $Y_{a,b}$, it contains at least a cusp of the Veech
group of $Y_{a,b}$.  This shows the estimate \eqref{eq:dioph}.

\begin{lemma}
\label{lem:cylinder-bounds}
Given a (finite-area) square-tiled surface $S$, there exists a
constant $K$ such that for every reduced rational $p/q \in (0,1]$,
cylinders in direction $p/q$ on $S$ have length at most $Kq$ and 
height at least $1/Kq$.
\end{lemma}

\begin{proof}
On the torus, the cylinder in direction $p/q$ has length 
$\sqrt{p^2+q^2}$ and height $1/\sqrt{p^2+q^2}$; assuming $p/q 
\leqslant1$, $\sqrt{p^2+q^2}\leqslant \sqrt{2}\cdot q$. Letting $d$ 
denote the number of squares of $S$, cylinders in direction $p/q$ on 
$S$ have length at most $d\cdot \sqrt{2}\cdot q$ and height at least 
$1/\sqrt{2} q$, which is no less than $1 / d\sqrt{2} q$.
\end{proof}

For a direction $p/q \in \PPP$, we proved that lifts to $X_{a,b}$ of
cylinders on $Y_{a,b}$ have double length and same height, so Lemma 
\ref{lem:cylinder-bounds} still holds for $X_{a,b}$ for directions in 
$\PPP$.

Let $R = R_{a,b}$ be the rectangle centered at the origin.  Fix
$\theta \in \Theta$, we want to prove that almost every orbit starting
from the boundary of $R$ comes back to $R$.  This claim implies that
the first return map on $R$ is defined almost everywhere.  As it is a
set of finite measure, by Poincaré recurrence Theorem, this map is
recurrent.  Then, by $\Z^2$ periodicity, the flow in direction
$\theta$ is recurrent in the table $T_{a,b}$.  Thus, it is enough to
prove the claim.

Let $p_{n}/ q_{n} \in \PPP$ such that
$$\lvert \theta - \frac{p_{n}}{q_{n}} \rvert \leq \frac{\e_{n}}{q_{n}^2}.$$

The boundary of the rectangle $R$ intersects transversally a finite
number of cylinders in direction $p_{n}/ q_{n}$.  This partitions
$\partial R$ into a finite number of intervals.  Calling $J$ one of
these intervals, and flowing in direction $\theta$ starting in $J$,
then after a time essentially equal to the length of the cylinder, the
trajectory crosses $J$ again, unless it leaves the cylinder before.
This only happens if the trajectory starts in an interval of measure
at most $K\e_{n}/q_{n}$ (see figure \ref{fig:recurrence}).  As the
length of $J$ is at least the reciprocal of the height of a cylinder
in direction $ \frac{p_{n}}{q_{n}}$, $\lvert J \rvert \geq 1/Kq_{n}$.
Consequently, the number of such intervals $J$ is at most $K'q_{n}$.
Thus, the measure of the points that do not return to $R$ is bounded
above by $KK'\e_{n}$.  As $\e_{n}$ tends to 0, almost every point
returns to $R$.  This ends the proof of Theorem \ref{main}.
\end{proofof}

\begin{figure}[ht]
\centering
\begin{tikzpicture} [scale=2]
  \draw [very thick] (0,0) -- (5,1) -- (5,2.4) -- (0,1.4) -- cycle;
  \draw [thick] (0,1.4) -- (5,2.25);
  \draw [thick] (0,1) -- (5,1.85);
  \draw [thick] (0,0.6) -- (5,1.45);
  \draw [thick] (0,0.2) -- (5,1.05);
  \draw [very thick,snake=brace,raise snake=3pt]
    (0,0) -- node [left=4pt] {$\leqslant K\e_n/q_n$} (0,0.2);
\end{tikzpicture}
\caption{Recurrence}
\label{fig:recurrence}
\end{figure}
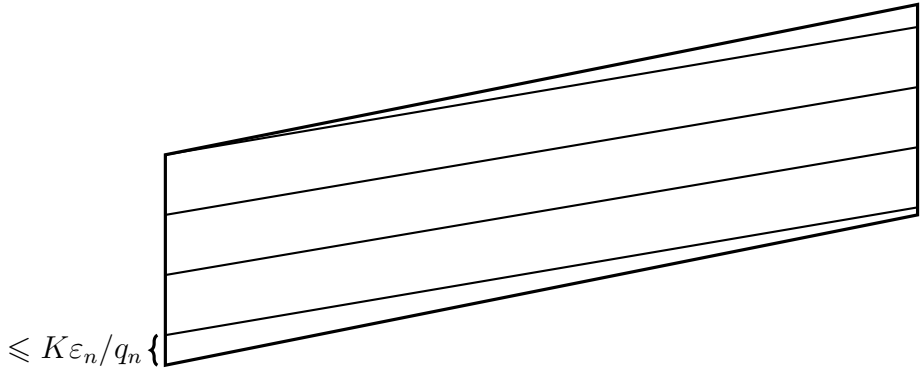

%\begin{figure}[ht]
%\centering
%\begin{tikzpicture} [scale=2]
%  \draw [very thick] (0,0) -- (5,1) -- (5,2.4) -- (0,1.4) -- cycle;
%  \draw [thick] (0,1.4) -- (5,2.25);
%  \draw [thick] (0,1) -- (5,1.85);
%  \draw [thick] (0,0.6) -- (5,1.45);
%  \draw [thick] (0,0.2) -- (5,1.05);
%  \draw [very thick,snake=brace,raise snake=3pt]
%    (0,0) -- node [left=4pt] {$\leqslant K\e_n/q_n$} (0,0.2);
%\end{tikzpicture}
%\caption{Recurrence}
%\label{fig:recurrence}
%\end{figure}

\section{Proof of Theorem \ref{remk:parity}}

\begin{proof}
Suppose $a = p/q$ and $b = r/s$ with $p$ and $r$ even.  Note that this
implies $q$ and $s$ are odd, and thus the surface $Y_{a,b}$ has an odd
number of squares.  $E$ and $F$ are integer Weierstrass points on
$Y_{a,b}$ because $p$ and $r$ are even.
If either $E$ of $F$ is on the waist line of a cylinder, this cylinder 
has even area.
Therefore neither $E$ nor $F$ can lie on the waist line of a cylinder
in a one-cylinder decomposition of $Y_{a,b}$.  Furthermore, $E$ and
$F$ cannot lie on the waist lines of two distinct cylinders in a
two-cylinder decomposition of $Y_{a,b}$.
It follows that in any rational direction, $Y_{a,b}$ has a cylinder
with two of the points $A$, $B$, $C$ on its waist line.  By
Lemma~\ref{lem:ABC-strip}, we conclude that in any rational direction
there is a strip on $X_{a,b}$.

The second point of the theorem is immediate: there are directions in
which $Y_{a,b}$ has only one cylinder; since these directions are not
completely periodic, this means that this cylinder of $Y_{a,b}$ lifts
to infinite strips on $X_{a,b}$; since this cylinder fills $Y_{a,b}$
these strips fill $X_{a,b}$, which means there is no periodic
trajectory on $X_{a,b}$.  This set of direction is dense and denoted
by $\PPP'$ in the sequel.

The proof of the third part has a similar structure to the proof of
Theorem \ref{main}.  Without loss of generality, we assume that the
slope is less than 1.  As we proved that $\PPP'$ contains at least a
cusp of the Veech group of $Y_{a,b}$ we can apply Khinchin-Sullivan's
Theorem (\cite{Su}), which states:

Let $\phi: \N \to \R^+$ be a function with
$(q\phi(q))$ non increasing and $\sum_{q=1}^{\infty} \phi(q) =
\infty$.  There exists a subset $\Theta$ of $[0,1]$ of full measure
satisfying: if $\theta$ belongs to $\Theta$, there exists a sequence
of rationals $(p_{n}/ q_{n}) \in \PPP'$ such that
$$\lvert q_{n}\theta - p_{n} \rvert \leq \phi(q_{n}).$$ 

Now, let $k \in \N$ and $f_{k}: x \mapsto \prod_{j=1}^k\log_{ j} x$, $\phi:
x \mapsto \e(x)/(xf_{k}(x))$ where $(\e(q))$ tends to zero as $q$
tends to infinity and is chosen so that $\sum_{q=1}^{\infty} \phi(q) =
\infty$.  For instance, one can choose $\e(x) = \frac{1}{\log_{ k+1}
x}$.

By Lemma \ref{lem:cylinder-bounds}, there is
a constant $K$ such that for every $p/ q \in \PPP'$, cylinders in
direction $p/q$ on $Y_{a,b}$ have length at most $Kq$ and height at
least $1/K q$.

% Since the cylinder lifts to strips, on the table $T_{a,b}$, we have
% $\dist(T^{Kq}x,x) \ge 1$ for all $x$ except for those in the boundary
% of the strips. ????

\begin{figure}[ht]
\begin{tikzpicture}[scale=0.4]
  \foreach \i in {-8,-2,4}
  \foreach \j in {-2,4}
    \draw [thick] (\i,\j) rectangle +(4,4);
  \fill [blue!20!white] (3.5,-1.5) -- ++(-0.5,0.5) -- ++(1,1) -- 
    ++(0,2) -- ++(-2,2) -- ++(-2,0) -- ++(-1,-1) -- ++(-2,2) --
    ++(0.5,0.5) -- ++(0,2) -- ++(-1.5,-1.5) -- ++(0,-2) --
    ++(2,-2) -- ++(2,0) -- ++(1,1) -- ++(2,-2) --
    ++(-1,-1) -- ++(0,-2) -- ++(1.5,-1.5) -- cycle;
%   \draw [thick,red] (3.5,-1.5) -- ++(-0.5,0.5) -- ++(1,1) -- 
%     ++(0,2) -- ++(-2,2) -- ++(-2,0) -- ++(-1,-1) -- ++(-2,2) --
%     ++(0.5,0.5) -- ++(0,2) -- ++(-1.5,-1.5) -- ++(0,-2) --
%     ++(2,-2) -- ++(2,0) -- ++(1,1) -- ++(2,-2) --
%     ++(-1,-1) -- ++(0,-2) -- ++(1.5,-1.5) -- cycle;
  \draw [very thin,blue] (3.5,-3.5) -- (2,-2) -- ++(2,2) --
    ++(-4,4) -- ++(-2,-2) -- ++(-2,2) -- ++(1.5,1.5);
  \draw [very thin,blue] (3.5,-1.5) -- (2,0) -- ++(2,2) --
    ++(-2,2) -- ++(-2,-2) -- ++(-4,4) -- ++(1.5,1.5);
  \draw [very thick,red,join=bevel] (3.5,-2.575) -- ++(-1.5,1.575)
    -- ++(2,2.1) -- ++(-2.7619,2.9) -- ++(-1.90476,-2) --
    ++(-1.90476,2) -- (-4,5.5) -- ++(1.5,1.575);
\end{tikzpicture}
\caption{Close to an escaping rational direction.
Here, a slope close to $1$, with $a = b = 2/3$.}
\label{fig:close-to-periodic}
\end{figure}

% 28 / 11 = 2.54545454545
% 20 / 11 = 1.81818181818
% 28 + 20 + 20 = 68
% 68 / 11 = 6.18181818182
% 8 - 6.181818 = 1.818182
% 1.81818 * 1.1 = 1.999998

% 2 * 1.075 = 2.15
% 1.5 * 1.075 = 1.6125
% 2.85 / 1.075 = 2.6511627907
% 2 / 1.075 = 1.86046511628

% 2 * 1.05 = 2.1
% 1.5 * 1.05 = 1.575
% 2.9 / 1.05 = 2.7619047619
% 2 / 1.05 = 1.90476190476
% 8 - 690 / 105 = 1.42857142857
% 1.42857142857 * 1.05 = 1.5

% 128 * 128 / (125 * 125) = 1.048576
% 125 * 125 / (128 * 128) = 0.953674316406

Let $R = R_{a,b}$ be the rectangle centered at the origin.  Fix
$\theta \in \Theta$, we want to prove that almost every orbit starting
from the boundary of $R$ satisfies $\limsup_{n \to \infty}
\dist(\phi^\theta_t x,x)/ \prod_{j=1}^k \log_j t = \infty$.  It is
enough to prove this claim by $\Z^2$ periodicity.

Let $p_{n}/ q_{n} \in \PPP'$ such that
$$\lvert \theta - \frac{p_{n}}{q_{n}} \rvert \leq
\frac{\e_{n}}{q_{n}^2 f_{k}(n) }.$$

The boundary of the rectangle $R$ intersects transversally a finite
number of strips in direction $p_{n}/ q_{n}$.  This partitions
$\partial R$ into a finite number of intervals.  Calling $J$ one of
these intervals, and flowing in direction $\theta$ starting in $J$,
then after a time essentially equal to $Kq_{n}$, the trajectory
crosses a copy of $J$ displaced by $(m_{1},m_{2}) \in \Z^2 \backslash
(0,0)$, unless it leaves the cylinder before, i.e.\ we have
$\dist(\phi^\theta_{Kq_{n}}x,x) \ge \sqrt{m_{1}^2 + m_{2}^2} \ge 1$.
This happens except if the trajectory starts in an interval of measure
at most $K\e_{n}/q_{n}f_k(q_{n})$ (see figure \ref{fig:recurrence}).
Repeat this $m$ times, since the cylinder in direction $p_{n}/ q_{n}$
lifts to a strip, we have
\begin{equation}\label{distance}
\dist(\phi^\theta_{mKq_{n}}x,x) \ge m
\end{equation} 
except for an interval of measure at most $mK\e_{n}/(q_{n} f_k(q_{n}))$.
Setting $m = f_k(q_{n})$ yields the upper bound $K\e_{n}/q_{n} $ on the
measure of the points that belong to $J$ and do not satisfy
\eqref{distance} with $m = f_k(q_{n})$.

As the length of $J$ is at least the height of a cylinder in direction
$ \frac{p_{n}}{q_{n}}$, $\lvert J \rvert \geq 1/Kq_{n}$.  Consequently,
the number of such intervals $J$ is at most $K'q_{n}$.  Thus, the
measure of the points of $R$ that do not satisfy \eqref{distance} with
$m = f_k(q_{n})$ is bounded above by $KK'\e_{n} \to 0$.
\end{proof}

\section{Proof of Corollary \ref{G-delta}}

In this section, we prove Corollary \ref{G-delta}.  We start with an 
easy lemma.

\begin{lemma} \label{density} The sets $\EEE$ and $\EEE'$
are both dense in $[0,1]\times [0,1]$.
\end{lemma}
\begin{proof}
To prove that $\EEE$ is dense in $[0,1]\times [0,1]$, it is enough to
show that the set
\begin{align*}
\Sigma & = 
\bigl\{\,  p/q \in \Q  :
 \quad 0<p<q,\quad p \textrm{ odd,}\quad q \textrm{ even}
\,\bigr\}
\end{align*}
is dense in $\Q \cap [0,1]$.  Let $a$ and $b$ integers with $(a,b) =
1$.  If $a/b$ is a rational in $[0,1]$.  If $a$, $b$ odd, the sequence
$(\frac{(2k+1)a}{(2k+1)b+a})$ of elements in $\Sigma$ tends to $a/b$.
If $a$ is even and $b$ odd then the sequence $(\frac{2k a}{2kb+a})$ of
elements in $\Sigma$ tends to $a/b$.  Thus the density is proven.
Essentially by the same proof, the set
\begin{align*}
\Sigma' & = 
\bigl\{\,  p/q \in \Q  :
\quad 0<p<q,\quad p \textrm{ even,}\quad q \textrm{ odd}
\,\bigr\}
\end{align*}
is dense in $\Q \cap [0,1]$, thus
$\EEE'$ is dense in $[0,1]\times [0,1]$.
\end{proof}

\begin{proofof}{Corollary \ref{G-delta}}
For each parameter value, call the rectangle centered at the origin
the \emph{base rectangle}.  Let $\Va$ denote the phase space of the
billiard map $\phi_{a,b}$ restricted to the base rectangle, i.e.\
$\Va$ is the direct product of boundary of the base rectangle with all
inward pointing unit vectors.

We will first prove, by approximating by tables from the set $\EEE
\cap \YYY$, that there is a dense $G_{\delta}$ set of tables
satisfying the first two points of the corollary.

\begin{claim*}
Fix a sequence
$\e_n \to 0$.  For each $(a,b) \in \EEE \cap \YYY$ and each $n \ge 1$,
we can choose a small open neighborhood $O(a,b,n)$ of $(a,b)$ such that
for all $(a',b') \in O(a,b,n)$,
\begin{enumerate}[label=\roman{*})\,]
\item{}
at least a proportion $(1-\e_n)$ of the points in $\Vp$ recur to 
$\Vp$,
\item{}
and the set of periodic points in $\Vp$ is at least $\e_n$-dense.
\end{enumerate}

\end{claim*}

First we show how the first two points of the corollary follow from
the claim.  For each $n \ge 1$ the set $\displaystyle O_n = \!\!\!\!\!
\bigcup_{(a,b) \in \EEE \cap \YYY} \!\!\!\!\!  O(a,b,n)$ contains
$\EEE \cap \YYY$ and thus is dense in $\YYY$.  Since $O_n$ is clearly
open, the set $\displaystyle \smash[t]{\bigcap_{m=1}^{\infty}
\bigcup_{n=m}^\infty O_n}$ is residual in $\YYY$.  For any parameter
$(a',b')$ in this set the set of periodic points in $\Vp$ are $\e_n$
dense for infinitely many $n$, i.e.\ they are dense in $\Vp$.
Furthermore for infinitely many $n$, a proportion $(1-\e_n)$ of the
points in $\Vp$ recur to $\Vp$, i.e.\ a.e.\ point in $\Vp$ recurs to
$\Vp$.  Since our tables are $\mathbb{Z}^2$ periodic, these two
results hold throughout the phase space, i.e. periodic points are
dense in the whole phase space (the second point of the corollary) and
a.e.\ point in a given rectangle recurs to that rectangle.  Thus for
any fixed rectangle, we can define almost everywhere a first return
map to that rectangle.  The first part of the corollary follow by
applying the Poincaré recurrence theorem and the Fubini theorem to
this map.

There remains to prove the claim.  Consider two parameter values
$(a_1,b_1)$ and $(a_2,b_2)$.  The corresponding base rectangles have
corners at $(\pm a_i/2, \pm b_i/2)$.  Throughout the rest of the proof
we will identify parts of the boundaries of the base rectangles as
follows: $(a_1,x)$ will be identified with $(a_2,x)$ for all $x$
satisfying $\lvert x\rvert \le \min(b_1,b_2)/2$, and similarly for the
other sides of the base rectangles.  This gives a natural
identification of the corresponding parts of the phase space $\Va$.
This identification respects the structure of the invariant measure,
thus it will be refered to as $\mu$, suppressing the parameter
dependence.

Fix $(a,b) \in \EEE \cap \YYY$.
Let $$O'(a,b,n) := \{(a',b') : \max(\lvert a'-a\rvert,\lvert
b'-b\rvert) \le \e_n/3\}.$$
Let $A = A(a,b,n)$ denote the part, which under the identification
described above, is common to $\Vp$ for all the parameter values in
$O'(a,b,n)$.
%The set $A$ is interpreted as being in $\Vp$ for all $(a',b') \in  O'(a,b,n)$.
For each $(a',b') \in \YYY$, viewed as a subset of $\Vp$, the set
$A$  satisfies
\begin{equation}\label{e1}
\mu(A) \ge  (1 - \e_n) \mu(\Vp) \text{  and } A \text{ is at least } \e_n
\text{-dense in } \Vp.
\end{equation}

Theorem \ref{billiard} implies that i) and ii) hold for $(a,b) \in
\EEE$.  We need to extend this to an open neighborhood of parameter
values.  For the moment we place ourselves in $\Va$.  We need to
identify a ``large'' set of recurrent and periodic points for which
the dynamics is identical for all $(a',b')$ in some open neighborhood
$O(a,b,n)$ at the time of periodicity or recurrence.

To do this we will remove orbits which approach too close to a corner.
For this first remove a $\delta_n$ neighborhood of the horizontal and
vertical directions in $A$.  Call this set $B \subset \Va$.  We choose
$\delta_n$ so small that
\begin{equation}\label{e2}
\mu(B \cap A) \ge (1 - \e_n) \mu(A) \text{ and } 
B \cap A \text{ is at least } \e_n \text{ dense in } A.
\end{equation}
Since we have removed a neighborhood of the horizontal and vertical
directions, there exists a positive constant $K_n$ (depending on
$\e_n$) such that the geometric length $|o|_g$ and the combinatorial
length $|o|_c$ of any finite orbit segment satisfies $|o|_g/|o|_c <
K_n$ for sufficiently long orbit segments.  The combinatorial length
of an orbit segment is the number of collisions it makes.

Let $C_N$ be the set of points in $B \cap A$ which recur to the base
rectangle before time $N$.  Note that periodic points recur at the
time of their period.  Theorem \ref{billiard} implies that $\mu(B \cap
A \setminus \displaystyle\cup_{n \in \N} C_N) = 0$.  Thus we can
choose $N=N(a,b,n)$ so large that
\begin{equation}\label{e3}
\mu(C_N) \ge (1 - \e_n) \mu(B \cap A) \text{ and } 
C_N \text{ is at least } \e_n \text{ dense in } B \cap A.
\end{equation}

The orbit segment of any point in $C_N$ until it recurs to the base
rectangle stays in a ball of radius $K_n N $ centered at the origin.
There are at most $K'_nN^2$ rectangles in this neighborhood and thus
at most $K_n'' N^2 $ corners (for certain constants $K_n'$ and
$K_n''$ which depends on $\e_n$).  Setting $\gamma_n = \e_n/(N K_n''
N^2)$, the $\gamma_n$-neighborhood of the corners of all these
rectangles has measure at most $K_n''N^2 \gamma_n = \e_n/N$.  Consider
the set $D := \{x \in C_N:$ there is no $ 1 \le i \le N$ such that
$\phi_{a,b}^ix$ is $\gamma_n$-close to a corner$\}$.  Then
\begin{equation}\label{e4}
\mu(D) \ge  (1 - \e_n) \mu(C_N) \text{ and } D \text{ is at least }
\e_n \text{-dense
in } C_N.
\end{equation}

Let $\beta_{n} >0$, $(a',b')$ such that $\max(\lvert 
a'-a\rvert,\lvert b'-b\rvert) \le
\beta_n$ and let $x$ be a point such that the orbit of $x$ has the
same combinatorics (intersects the same rectangles) until the $N$th
collision on the tables $T_{a,b}$ and $T_{a',b'}$.  We have: $\lvert
\phi_{a,b}^k(x) - \phi_{a',b'}^k(x)\rvert < K'''N \beta_{n}$
where $K'''$ is a constant (this means that the orbits diverge at
speed at most linear which is obvious).  Take $\beta_{n}$ so small
that $2N\beta_{n} \leq \gamma_{n}/2$.  Now take $O(a,b,n) \subset
O'(a,b,n)$ a small neighborhood of $(a,b)$ given by $O(a,b,n) :=
\{(a',b') : \max(|a'-a|,|b'-b|) \le \beta_n\}$.  On the set $D$, the
dynamics for any table $(a',b') \in O(a,b,n)$ has the same
combinatorics as the dynamics on the table $(a,b)$.  In particular all
recurrent points of $D$ are recurrent for all parameter values.

To get the same result for periodic orbits, we need the following

\begin{lemma} \label{stability}
Periodic orbits are stable under small perturbations of the parameters.
\end{lemma}

\begin{proof}
Given a periodic orbit $\gamma$ of slope $\theta$ and combinatorial
length $N$ on $T_{a,b}$, we want to construct a periodic orbit
$\gamma'$ of the same combinatorial length on $T_{a' ,b'}$ where
$(a',b')$ belongs to a small neighborhood of $(a,b)$.  We also want
the slope and the starting point of $\gamma'$ to be close to those of
$\gamma$.

Our periodic orbit belongs to a cylinder $\CCC$ of height $2d$.
Without loss of generality, we assume that the periodic orbit is the
waist curve of this cylinder.  We fix the origin $x$ on a vertical
side of a rectangle.  To get estimates, we also assume that $\theta$
belongs to the complement of a small neighborhood of the horizontal
and vertical directions.  We take $(a',b')$ close enough from $(a,b)$
so that the orbits of $x$ of slope $\theta$ in $T_{a,b}$ and
$T_{a',b'}$ have the same combinatorics (exactly as in the proof of
the previous lemma).  We denote by $O(a,b)$ this neighborhood and
$\delta$ its diameter.  Call $J$ the vertical segment containing $x$
(the common part under identifications for all the parameter values in
$O(a,b)$).  We unfold the trajectory $\gamma$ starting from $x$.
Denote by $J_{a,b}$ and $J_{a',b'}$ the segments obtained after $N$
collisions.  As we already mentioned in the previous lemma,
$J_{a',b'}$ is translated from $J_{a,b}$ by a vector of length at most
$\kappa N\delta$ where $\kappa$ is a constant depending on the slope
(uniform when $\theta$ is not too close from the vertical and
horizontal directions).  We choose $\delta$ so small that $\kappa
N\delta < d/2$.  Denote by $x_{a',b'}$, the point $x$ on $J_{a',b'}$.
The point $x_{a',b'}$ belongs to the strip of height $d$ contained in
the unfolding of the cylinder $\CCC$ (see figure
\ref{fig:stable-periodic}).  If $\delta$ is small enough, no
singularity enters this strip.  Thus, the projection of the segment
from $x$ to $x_{a',b'}$ is a periodic orbit on $T_{a',b'}$ with the
required properties.
\end{proof}

% \begin{figure}[ht]
% \begin{tikzpicture}
%   % windtree, a = b = 1/2, slope = 1/1
%   [x=1.75cm,y=1.75cm,scale=1.2]
%   \foreach \i in {-2,0,...,2}
%   \foreach \j in {0,2,...,4}
%     \draw [thick,red] (\i,\j) rectangle +(1,1);
%   \draw[very thick,red] (0.5,1) -- ++(1.5,1.5) -- ++(-1.5,1.5)
%      -- ++(-1.5,-1.5) -- ++(1.5,-1.5);
%   \foreach \i in {-2,0,...,2}
%   \foreach \j in {0,2,...,4}
%     \draw[thick,blue] (\i,\j) ++(0.05,-0.05) rectangle +(0.9,1.1);
%   \draw[very thick,blue] (0.5,1.05) -- ++(1.55,1.55) -- ++(-1.35,1.35)
%      -- ++(-1.75,-1.75) -- ++(1.15,-1.15);
% \end{tikzpicture}
% \caption{Stability of periodic orbits.}
% \label{fig:stable-periodic}
% \end{figure}

% \begin{figure}[ht]
% \begin{tikzpicture}
%   % windtree, a = b = 1/2, slope = 1/1
%   [x=1.75cm,y=1.75cm,scale=1.2]
%   \foreach \i in {-2,0,...,2}
%   \foreach \j in {0,2,...,4}
%     \draw [thick,blue] (\i,\j) rectangle +(1,1);
%   \draw[very thick,blue] (0.6,1) -- ++(1.4,1.4) -- ++(-1.6,1.6)
%      -- ++(-1.4,-1.4) -- ++(1.6,-1.6);
%   \foreach \i in {-2,0,...,2}
%   \foreach \j in {0,2,...,4}
%     \draw[thick,red] (\i,\j) ++(0.05,-0.05) rectangle +(0.9,1.1);
%   \draw[very thick,red] (0.6,1.05) -- ++(1.45,1.45) -- ++(-1.45,1.45)
%      -- ++(-1.65,-1.65) -- ++(1.25,-1.25);
% \end{tikzpicture}
% \caption{Stability of periodic orbits.}
% \label{fig:stable-periodic}
% \end{figure}

\begin{figure}[ht]
\begin{tikzpicture}
  % windtree, a = b = 1/2, slope = 1/1
  [x=1.75cm,y=1.75cm,scale=1.2]
  \foreach \i in {-2,0,...,2}
  \foreach \j in {0,2,...,4}
    \draw [thick,blue] (\i,\j) rectangle +(1,1);
  \draw[very thick,blue] (-1,2.5) -- ++(1.5,1.5) -- ++(1.5,-1.5)
     -- ++(-1.5,-1.5) -- ++(-1.5,1.5);
  \foreach \i in {-2,0,...,2}
  \foreach \j in {0,2,...,4}
    \draw[thick,red] (\i,\j) ++(0.05,-0.05) rectangle +(0.9,1.1);
  \draw[very thick,red] (-1.05,2.5) -- ++(1.45,1.45) -- ++(1.65,-1.65)
     -- ++(-1.25,-1.25) -- ++(-1.85,1.85);
  \draw [thin,->] (-0.8,1.8) node [below] {$x$} to
    [bend right=5] (-0.95,2.35);
  \draw [thin,->] (-1.35,2.3) node [left] {$x_{a',b'}$} +(-0.05,0.1) to
    [bend left=5] (-1.1,2.5);
\end{tikzpicture}
\caption{Stability of periodic orbits.}
\label{fig:stable-periodic}
\end{figure}
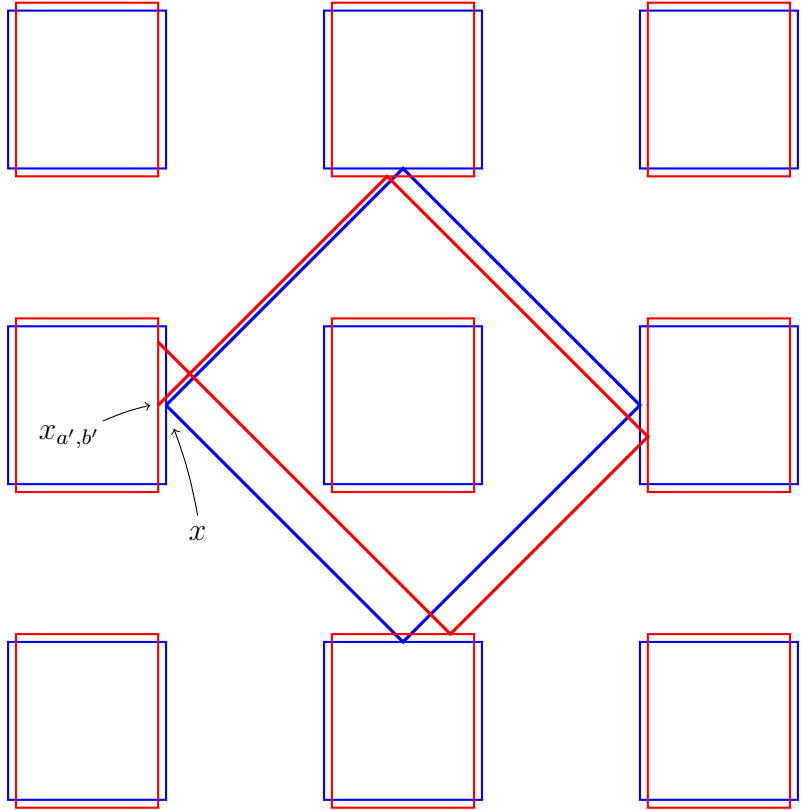

Combining  Equations
\eqref{e1}-\eqref{e4} yields (after appropriate redefinition of the
$\e$'s) our claim:
\begin{equation}
\mu(D) \ge (1-\e_n)^4 \mu(\Vp) \text{ and } D \text{ is at least }
4 \e_n \text{ dense
in } \Vp.
\end{equation}
for all $(a',b') \in O(a,b,n)$.

\medskip

Finally we will construct a dense $G_{\delta}$ set of tables which
satisfies the third point of the corollary.  Fix a sequence $\e_n \to
0$ and fix $k \ge 1$.  

\medskip

\begin{claim*}
For each $(a,b) \in \EEE' \cap \YYY$ and each
$n \ge 1$  we can choose a small open neighborhood
$O(a,b,n)$ of $(a,b) \in \EEE'$ such that, for $(a',b') \in 
O(a,b,n)$,
\begin{enumerate}[label=\roman{*})\,]
\setcounter{enumi}{2}
\item{}
there exists $m = m(a,b,n,x)$ such that at least a proportion
$(1-\e_n)$ of the points in $\Vp$ satisfy $\displaystyle
\frac{\dist(\phi^\theta_mx,x)}{\prod_{j=1}^k \log_{j} m} \ge n$.
\end{enumerate}
\end{claim*}
\medskip

We next show that the corollary follows from the claim.  For each $n
\ge 1$ the set $\displaystyle O_n = \!\!\!\!\!  \bigcup_{(a,b) \in
\EEE' \cap \YYY} \!\!\!\!\!  O(a,b,n)$ contains $\EEE' \cap \YYY$ and
thus is dense in $\YYY$.  Since $O_n$ is clearly open the set
$\displaystyle \smash[t]{\bigcap_{m=1}^{\infty} \bigcup_{n=m}^\infty
O_n}$ is residual in $\YYY$.  For any parameter $(a',b')$ in this set
we have for a.e.\ point in $\Vp$ for any infinite subsequence of $n$'s
there are $a = a(n)$, $b = b(n)$ such that $a',b' \in O(a,b,n)$ and
there is $m = m(a,b,n,x)$ such that
$\dist(\phi^\theta_mx,x)/(\prod_{j=1}^k \log_{j}m) \ge n$ for a.e.\
point in $\Vp$.  Since our tables are $\mathbb{Z}^2$ periodic, there
is a dense $G_{\delta}$ of tables satisfying the third point in the
corollary.  Since the intersection of two dense $G_{\delta}$ sets is
again a dense $G_{\delta}$ the corollary follows.

There remains to prove the claim.  Fix $(a,b) \in \EEE' \cap \YYY$ and
define as above the sets $O'(a,b,n)$,and $A = A(a,b,n)$.  Theorem
\ref{remk:parity} implies that iii) holds for $(a,b) \in \EEE'$.  We
need to extend this to an open neighborhood of parameter values.  For
the moment we place ourselves in $\Va$.  As before we remove a
$\delta_n$ neighborhood of the horizontal and vertical directions in
$A$, calling the resulting set $B \subset \Va$.  We choose $\delta_n$
so small that equation \ref{e2} is satisfied and choose $K_n$ as
before as well.

Let $C_N$ be the set of points $x \in B \cap A$ which satisfy iii)
with $m(a,b,n,x) \le N$.  Theorem \ref{remk:parity} implies that
$\mu(B \cap A \setminus \displaystyle\cup_{n \in \N} C_N) = 0$.  Thus
we can choose $N=N(a,b,n)$ so large that
\begin{equation}\label{e3'}
\mu(C_N) \ge (1 - \e_n) \mu(B \cap A).
\end{equation}

The orbit segment up to time $m(a,b,n,x)$ of any point $x \in C_N$
stays in a ball of radius $K_n N $ centered at the origin.  Defining
$D$ (the set of points in $C_N$ whose orbits of length $N$ stay
sufficiently far away from corners to have the same symbolic orbit for
all parameter values in $O(a,b,n)$) in the identical way as above
yields
%There are at most $K'_nN^2$ rectangles in this
%neighborhood and thus at most $K_n'' N^2 $ corners (for a certain constants
%$K_n'$ and $K_n''$ which depends on $\e_n$). 
%Take  $\gamma_n := \e_n/(N K_n'' N^2)$ neighborhood of the corners of all
%these rectangles, it has measure at most $K_n''N^2 \gamma_n = \e_n/N$. Consider
%the set  $D := \{x \in C_N:$ there is no $ 1 \le i \le N$ such that  $\phi_{a,b}^ix$
%is in a $\gamma_n$ neighborhood of a corner$\}$.  Then
\begin{equation}\label{e4'}
\mu(D) \ge  (1 - \e_n) \mu(C_N).
\end{equation}

%Let  $\beta_{n} >0$, $(a',b')$ such that $\max(|a'-a|,|b'-b|) \le \beta_n$ and let  $x$ be a point such that the orbit of $x$ has the same combinatorics (intersects the same rectangles) until the $N$th collision on the tables $T_{a,b}$ and $T_{a',b'}$.   We have:
%$\vert \phi_{a,b}^k(x) -  \vert \phi_{a',b'}^k(x)\vert < K'''N \beta_{n}$ where $K'''$ is a constant (this means that the orbits diverge at speed at most linear which is obvious). Take $\beta_{n}$ so small that $2N\beta_{n} \leq \gamma_{n}/2$.
%Now take $O(a,b,n) \subset O'(a,b,n)$ a small neighborhood of $(a,b)$
%given by $O(a,b,n) := \{(a',b') : \max(|a'-a|,|b'-b|) \le \beta_n\}$.
%On the set $D$,
%the dynamics for any table $(a',b') \in O(a,b,n)$ has the same combinatorics as
%the dynamics on the table $(a,b)$. In particular all the points which
%satisfy iii) in $D$ do so for all parameter values.

Combining Equations \eqref{e1},\eqref{e2},\eqref{e3'},\eqref{e4'}
yields (after appropriate redefinition of the $\e$'s) our claim:
for all $(a',b') \in O(a,b,n)$,
\begin{equation}
\mu(D) \ge (1-\e_n)^4 \mu(\Vp).
\end{equation}

\vspace*{-12pt}
\end{proofof}

% \newpage

\section*{Appendix: a few pictures}

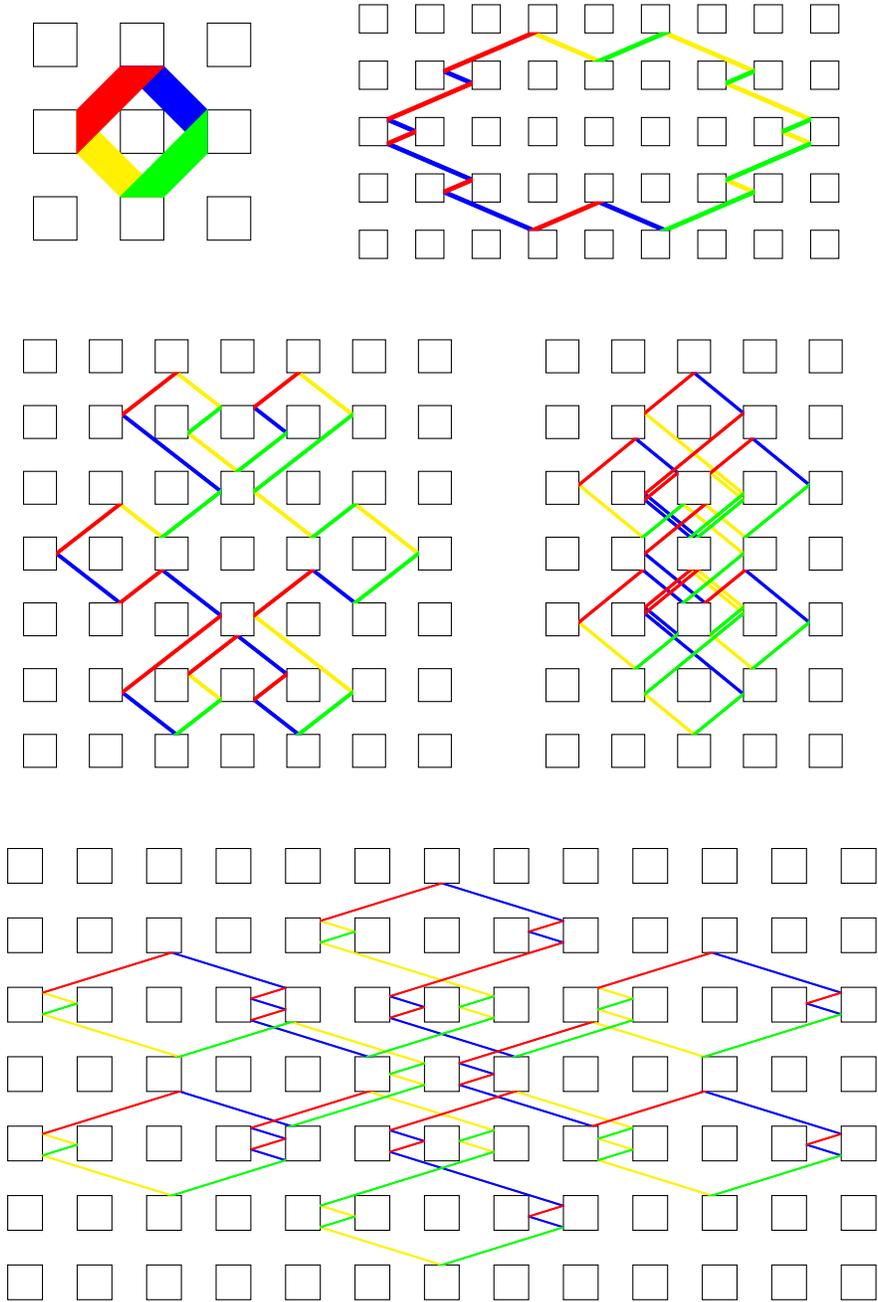
\begin{figure}[ht]
\begin{minipage}[ht]{0.33\linewidth}
\centering
\begin{tikzpicture}
  % windtree, a = b = 1/2, slope = 1/1
  \begin{scope}
  [x=1.75cm,y=1.75cm,scale=0.33]
  \foreach \i in {-2,0,...,2}
  \foreach \j in {0,2,...,4}
    \draw[thin] (\i,\j) rectangle +(1,1);
  \draw[very thin,fill,color=blue] (2,2) -- +(0,1) -- (1,4) -- +(-1,0) -- cycle;
  \draw[very thin,fill,color=yellow] (-1,2) -- +(0,1) -- (1,1) -- +(-1,0) -- cycle;
  \draw[very thin,fill,color=green] (0,1) -- +(1,0) -- (2,2) -- +(0,1) -- cycle;
  \draw[very thin,fill,color=red] (0,4) -- +(1,0) -- (-1,2) -- +(0,1) -- cycle;
  \end{scope}
\end{tikzpicture}
\end{minipage}%
\begin{minipage}[ht]{0.63\linewidth}
\centering
\begin{tikzpicture}
  % windtree, a = b = 1/2, slope = 3/7
  [x=0.416666666667cm,y=0.178571428571cm,scale=0.3]
  \foreach \i in {-24,-18,...,24}
  \foreach \j in {-42,-28,...,14}
    \draw[thin] (\i,\j) rectangle +(3,7);
  \draw[very thin,fill,color=yellow] (8,14) -- +(1,0) -- (18,5) -- +(0,-1) -- cycle;
  \draw[very thin,fill,color=yellow] (15,1) -- +(0,1) -- (24,-7) -- +(0,-1) -- cycle;
  \draw[very thin,fill,color=yellow] (21,-11) -- +(0,1) -- (24,-13) -- +(0,-1) -- cycle;
  \draw[very thin,fill,color=yellow] (15,-23) -- +(0,1) -- (18,-25) -- +(0,-1) -- cycle;
  \draw[very thin,fill,color=blue] (8,-35) -- +(1,0) -- (2,-28) -- +(-1,0) -- cycle;
  \draw[very thin,fill,color=blue] (-6,-35) -- +(1,0) -- (-15,-25) -- +(0,-1) -- cycle;
  \draw[very thin,fill,color=blue] (-12,-23) -- +(0,1) -- (-21,-13) -- +(0,-1) -- cycle;
  \draw[very thin,fill,color=blue] (-18,-11) -- +(0,1) -- (-21,-7) -- +(0,-1) -- cycle;
  \draw[very thin,fill,color=blue] (-12,1) -- +(0,1) -- (-15,5) -- +(0,-1) -- cycle;
  \draw[very thin,fill,color=yellow] (-6,14) -- +(1,0) -- (2,7) -- +(-1,0) -- cycle;
  \draw[very thin,fill,color=green] (1,7) -- +(1,0) -- (9,14) -- +(-1,0) -- cycle;
  \draw[very thin,fill,color=green] (18,4) -- +(0,1) -- (15,2) -- +(0,-1) -- cycle;
  \draw[very thin,fill,color=green] (24,-8) -- +(0,1) -- (21,-10) -- +(0,-1) -- cycle;
  \draw[very thin,fill,color=green] (24,-14) -- +(0,1) -- (15,-22) -- +(0,-1) -- cycle;
  \draw[very thin,fill,color=green] (18,-26) -- +(0,1) -- (8,-35) -- +(1,0) -- cycle;
  \draw[very thin,fill,color=red] (1,-28) -- +(1,0) -- (-5,-35) -- +(-1,0) -- cycle;
  \draw[very thin,fill,color=red] (-15,-26) -- +(0,1) -- (-12,-22) -- +(0,-1) -- cycle;
  \draw[very thin,fill,color=red] (-21,-14) -- +(0,1) -- (-18,-10) -- +(0,-1) -- cycle;
  \draw[very thin,fill,color=red] (-21,-8) -- +(0,1) -- (-12,2) -- +(0,-1) -- cycle;
  \draw[very thin,fill,color=red] (-15,4) -- +(0,1) -- (-6,14) -- +(1,0) -- cycle;
\end{tikzpicture}
\end{minipage}

\bigskip
\bigskip
\bigskip

\begin{minipage}[ht]{0.53\linewidth}
\centering
\begin{tikzpicture}
  % windtree, a = b = 1/2, slope = 7/9
  [x=0.138888888889cm,y=0.108024691358cm,scale=0.45]
  \foreach \i in {-42,-28,...,42}
  \foreach \j in {-72,-54,...,36}
    \draw[thin] (\i,\j) rectangle +(7,9);
  \draw[very thin,fill,color=blue] (14,19) -- +(0,1) -- (7,27) -- +(0,-1) -- cycle;
  \draw[very thin,fill,color=yellow] (16,36) -- +(1,0) -- (28,25) -- +(0,-1) -- cycle;
  \draw[very thin,fill,color=yellow] (7,3) -- +(0,1) -- (20,-9) -- +(-1,0) -- cycle;
  \draw[very thin,fill,color=yellow] (28,0) -- +(1,0) -- (42,-13) -- +(0,-1) -- cycle;
  \draw[very thin,fill,color=blue] (28,-27) -- +(1,0) -- (20,-18) -- +(-1,0) -- cycle;
  \draw[very thin,fill,color=yellow] (7,-31) -- +(0,1) -- (28,-51) -- +(0,-1) -- cycle;
  \draw[very thin,fill,color=blue] (16,-63) -- +(1,0) -- (7,-53) -- +(0,-1) -- cycle;
  \draw[very thin,fill,color=blue] (14,-47) -- +(0,1) -- (4,-36) -- +(-1,0) -- cycle;
  \draw[very thin,fill,color=yellow] (-7,-47) -- +(0,1) -- (0,-53) -- +(0,-1) -- cycle;
  \draw[very thin,fill,color=blue] (-10,-63) -- +(1,0) -- (-21,-51) -- +(0,-1) -- cycle;
  \draw[very thin,fill,color=blue] (0,-31) -- +(0,1) -- (-12,-18) -- +(-1,0) -- cycle;
  \draw[very thin,fill,color=blue] (-22,-27) -- +(1,0) -- (-35,-13) -- +(0,-1) -- cycle;
  \draw[very thin,fill,color=yellow] (-22,0) -- +(1,0) -- (-12,-9) -- +(-1,0) -- cycle;
  \draw[very thin,fill,color=blue] (0,3) -- +(0,1) -- (-21,25) -- +(0,-1) -- cycle;
  \draw[very thin,fill,color=yellow] (-10,36) -- +(1,0) -- (0,27) -- +(0,-1) -- cycle;
  \draw[very thin,fill,color=yellow] (-7,19) -- +(0,1) -- (4,9) -- +(-1,0) -- cycle;
  \draw[very thin,fill,color=green] (3,9) -- +(1,0) -- (14,19) -- +(0,1) -- cycle;
  \draw[very thin,fill,color=red] (7,26) -- +(0,1) -- (16,36) -- +(1,0) -- cycle;
  \draw[very thin,fill,color=green] (28,24) -- +(0,1) -- (7,4) -- +(0,-1) -- cycle;
  \draw[very thin,fill,color=green] (19,-9) -- +(1,0) -- (29,0) -- +(-1,0) -- cycle;
  \draw[very thin,fill,color=green] (42,-14) -- +(0,1) -- (28,-27) -- +(1,0) -- cycle;
  \draw[very thin,fill,color=red] (19,-18) -- +(1,0) -- (7,-31) -- +(0,1) -- cycle;
  \draw[very thin,fill,color=green] (28,-52) -- +(0,1) -- (16,-63) -- +(1,0) -- cycle;
  \draw[very thin,fill,color=red] (7,-54) -- +(0,1) -- (14,-46) -- +(0,-1) -- cycle;
  \draw[very thin,fill,color=red] (3,-36) -- +(1,0) -- (-7,-47) -- +(0,1) -- cycle;
  \draw[very thin,fill,color=green] (0,-54) -- +(0,1) -- (-10,-63) -- +(1,0) -- cycle;
  \draw[very thin,fill,color=red] (-21,-52) -- +(0,1) -- (0,-30) -- +(0,-1) -- cycle;
  \draw[very thin,fill,color=red] (-13,-18) -- +(1,0) -- (-21,-27) -- +(-1,0) -- cycle;
  \draw[very thin,fill,color=red] (-35,-14) -- +(0,1) -- (-22,0) -- +(1,0) -- cycle;
  \draw[very thin,fill,color=green] (-13,-9) -- +(1,0) -- (0,3) -- +(0,1) -- cycle;
  \draw[very thin,fill,color=red] (-21,24) -- +(0,1) -- (-10,36) -- +(1,0) -- cycle;
  \draw[very thin,fill,color=green] (0,26) -- +(0,1) -- (-7,20) -- +(0,-1) -- cycle;
\end{tikzpicture}
\end{minipage}%
\begin{minipage}[ht]{0.43\linewidth}
\centering
\begin{tikzpicture}
  % windtree, a = b = 1/2, slope = 9/11
  [x=0.108024691358cm,y=0.0883838383838cm,scale=0.45]
  \foreach \i in {-36,-18,...,36}
  \foreach \j in {0,22,...,132}
    \draw[thin] (\i,\j) rectangle +(9,11);
  \draw[very thin,fill,color=blue] (18,24) -- +(0,1) -- (-9,52) -- +(0,-1) -- cycle;
  \draw[very thin,fill,color=yellow] (5,66) -- +(1,0) -- (18,54) -- +(0,-1) -- cycle;
  \draw[very thin,fill,color=yellow] (9,44) -- +(0,1) -- (21,33) -- +(-1,0) -- cycle;
  \draw[very thin,fill,color=blue] (36,48) -- +(0,1) -- (19,66) -- +(-1,0) -- cycle;
  \draw[very thin,fill,color=blue] (7,55) -- +(1,0) -- (-9,72) -- +(0,-1) -- cycle;
  \draw[very thin,fill,color=yellow] (7,88) -- +(1,0) -- (19,77) -- +(-1,0) -- cycle;
  \draw[very thin,fill,color=blue] (36,94) -- +(0,1) -- (21,110) -- +(-1,0) -- cycle;
  \draw[very thin,fill,color=yellow] (9,98) -- +(0,1) -- (18,90) -- +(0,-1) -- cycle;
  \draw[very thin,fill,color=blue] (5,77) -- +(1,0) -- (-9,92) -- +(0,-1) -- cycle;
  \draw[very thin,fill,color=blue] (18,118) -- +(0,1) -- (5,132) -- +(-1,0) -- cycle;
  \draw[very thin,fill,color=yellow] (-9,118) -- +(0,1) -- (18,92) -- +(0,-1) -- cycle;
  \draw[very thin,fill,color=blue] (3,77) -- +(1,0) -- (-9,90) -- +(0,-1) -- cycle;
  \draw[very thin,fill,color=blue] (0,98) -- +(0,1) -- (-11,110) -- +(-1,0) -- cycle;
  \draw[very thin,fill,color=yellow] (-27,94) -- +(0,1) -- (-9,77) -- +(-1,0) -- cycle;
  \draw[very thin,fill,color=yellow] (1,88) -- +(1,0) -- (18,72) -- +(0,-1) -- cycle;
  \draw[very thin,fill,color=blue] (1,55) -- +(1,0) -- (-9,66) -- +(-1,0) -- cycle;
  \draw[very thin,fill,color=yellow] (-27,48) -- +(0,1) -- (-11,33) -- +(-1,0) -- cycle;
  \draw[very thin,fill,color=blue] (0,44) -- +(0,1) -- (-9,54) -- +(0,-1) -- cycle;
  \draw[very thin,fill,color=yellow] (3,66) -- +(1,0) -- (18,52) -- +(0,-1) -- cycle;
  \draw[very thin,fill,color=yellow] (-9,24) -- +(0,1) -- (5,11) -- +(-1,0) -- cycle;
  \draw[very thin,fill,color=green] (4,11) -- +(1,0) -- (18,24) -- +(0,1) -- cycle;
  \draw[very thin,fill,color=red] (-9,51) -- +(0,1) -- (5,66) -- +(1,0) -- cycle;
  \draw[very thin,fill,color=green] (18,53) -- +(0,1) -- (9,45) -- +(0,-1) -- cycle;
  \draw[very thin,fill,color=green] (20,33) -- +(1,0) -- (36,48) -- +(0,1) -- cycle;
  \draw[very thin,fill,color=red] (18,66) -- +(1,0) -- (8,55) -- +(-1,0) -- cycle;
  \draw[very thin,fill,color=red] (-9,71) -- +(0,1) -- (7,88) -- +(1,0) -- cycle;
  \draw[very thin,fill,color=green] (18,77) -- +(1,0) -- (36,94) -- +(0,1) -- cycle;
  \draw[very thin,fill,color=red] (20,110) -- +(1,0) -- (9,98) -- +(0,1) -- cycle;
  \draw[very thin,fill,color=green] (18,89) -- +(0,1) -- (5,77) -- +(1,0) -- cycle;
  \draw[very thin,fill,color=red] (-9,91) -- +(0,1) -- (18,119) -- +(0,-1) -- cycle;
  \draw[very thin,fill,color=red] (4,132) -- +(1,0) -- (-9,118) -- +(0,1) -- cycle;
  \draw[very thin,fill,color=green] (18,91) -- +(0,1) -- (3,77) -- +(1,0) -- cycle;
  \draw[very thin,fill,color=red] (-9,89) -- +(0,1) -- (0,99) -- +(0,-1) -- cycle;
  \draw[very thin,fill,color=red] (-12,110) -- +(1,0) -- (-27,94) -- +(0,1) -- cycle;
  \draw[very thin,fill,color=green] (-10,77) -- +(1,0) -- (2,88) -- +(-1,0) -- cycle;
  \draw[very thin,fill,color=green] (18,71) -- +(0,1) -- (1,55) -- +(1,0) -- cycle;
  \draw[very thin,fill,color=red] (-10,66) -- +(1,0) -- (-27,48) -- +(0,1) -- cycle;
  \draw[very thin,fill,color=green] (-12,33) -- +(1,0) -- (0,44) -- +(0,1) -- cycle;
  \draw[very thin,fill,color=red] (-9,53) -- +(0,1) -- (3,66) -- +(1,0) -- cycle;
  \draw[very thin,fill,color=green] (18,51) -- +(0,1) -- (-9,25) -- +(0,-1) -- cycle;
\end{tikzpicture}
\end{minipage}

\bigskip
\bigskip
\bigskip

\begin{minipage}[ht]{0.98\linewidth}
\centering
\begin{tikzpicture}
  % windtree, a = b = 1/2, slope = 9/29
  [x=0.102592592593cm,y=0.0318390804598cm,scale=0.5]
  \foreach \i in {-108,-90,...,108}
  \foreach \j in {0,58,...,348}
    \draw[thin] (\i,\j) rectangle +(9,29);
  \draw[very thin,fill,color=blue] (36,60) -- +(0,1) -- (27,70) -- +(0,-1) -- cycle;
  \draw[very thin,fill,color=blue] (36,78) -- +(0,1) -- (-9,124) -- +(0,-1) -- cycle;
  \draw[very thin,fill,color=blue] (0,132) -- +(0,1) -- (-9,142) -- +(0,-1) -- cycle;
  \draw[very thin,fill,color=yellow] (23,174) -- +(1,0) -- (54,144) -- +(0,-1) -- cycle;
  \draw[very thin,fill,color=yellow] (45,134) -- +(0,1) -- (54,126) -- +(0,-1) -- cycle;
  \draw[very thin,fill,color=yellow] (45,116) -- +(0,1) -- (75,87) -- +(-1,0) -- cycle;
  \draw[very thin,fill,color=blue] (108,120) -- +(0,1) -- (99,130) -- +(0,-1) -- cycle;
  \draw[very thin,fill,color=blue] (108,138) -- +(0,1) -- (73,174) -- +(-1,0) -- cycle;
  \draw[very thin,fill,color=blue] (43,145) -- +(1,0) -- (9,180) -- +(0,-1) -- cycle;
  \draw[very thin,fill,color=blue] (18,188) -- +(0,1) -- (9,198) -- +(0,-1) -- cycle;
  \draw[very thin,fill,color=yellow] (43,232) -- +(1,0) -- (73,203) -- +(-1,0) -- cycle;
  \draw[very thin,fill,color=blue] (108,238) -- +(0,1) -- (99,248) -- +(0,-1) -- cycle;
  \draw[very thin,fill,color=blue] (108,256) -- +(0,1) -- (75,290) -- +(-1,0) -- cycle;
  \draw[very thin,fill,color=yellow] (45,260) -- +(0,1) -- (54,252) -- +(0,-1) -- cycle;
  \draw[very thin,fill,color=yellow] (45,242) -- +(0,1) -- (54,234) -- +(0,-1) -- cycle;
  \draw[very thin,fill,color=blue] (23,203) -- +(1,0) -- (-9,236) -- +(0,-1) -- cycle;
  \draw[very thin,fill,color=blue] (0,244) -- +(0,1) -- (-9,254) -- +(0,-1) -- cycle;
  \draw[very thin,fill,color=blue] (36,298) -- +(0,1) -- (27,308) -- +(0,-1) -- cycle;
  \draw[very thin,fill,color=blue] (36,316) -- +(0,1) -- (5,348) -- +(-1,0) -- cycle;
  \draw[very thin,fill,color=yellow] (-27,316) -- +(0,1) -- (-18,308) -- +(0,-1) -- cycle;
  \draw[very thin,fill,color=yellow] (-27,298) -- +(0,1) -- (18,254) -- +(0,-1) -- cycle;
  \draw[very thin,fill,color=yellow] (9,244) -- +(0,1) -- (18,236) -- +(0,-1) -- cycle;
  \draw[very thin,fill,color=blue] (-15,203) -- +(1,0) -- (-45,234) -- +(0,-1) -- cycle;
  \draw[very thin,fill,color=blue] (-36,242) -- +(0,1) -- (-45,252) -- +(0,-1) -- cycle;
  \draw[very thin,fill,color=blue] (-36,260) -- +(0,1) -- (-65,290) -- +(-1,0) -- cycle;
  \draw[very thin,fill,color=yellow] (-99,256) -- +(0,1) -- (-90,248) -- +(0,-1) -- cycle;
  \draw[very thin,fill,color=yellow] (-99,238) -- +(0,1) -- (-63,203) -- +(-1,0) -- cycle;
  \draw[very thin,fill,color=yellow] (-35,232) -- +(1,0) -- (0,198) -- +(0,-1) -- cycle;
  \draw[very thin,fill,color=yellow] (-9,188) -- +(0,1) -- (0,180) -- +(0,-1) -- cycle;
  \draw[very thin,fill,color=blue] (-35,145) -- +(1,0) -- (-63,174) -- +(-1,0) -- cycle;
  \draw[very thin,fill,color=yellow] (-99,138) -- +(0,1) -- (-90,130) -- +(0,-1) -- cycle;
  \draw[very thin,fill,color=yellow] (-99,120) -- +(0,1) -- (-65,87) -- +(-1,0) -- cycle;
  \draw[very thin,fill,color=blue] (-36,116) -- +(0,1) -- (-45,126) -- +(0,-1) -- cycle;
  \draw[very thin,fill,color=blue] (-36,134) -- +(0,1) -- (-45,144) -- +(0,-1) -- cycle;
  \draw[very thin,fill,color=yellow] (-15,174) -- +(1,0) -- (18,142) -- +(0,-1) -- cycle;
  \draw[very thin,fill,color=yellow] (9,132) -- +(0,1) -- (18,124) -- +(0,-1) -- cycle;
  \draw[very thin,fill,color=yellow] (-27,78) -- +(0,1) -- (-18,70) -- +(0,-1) -- cycle;
  \draw[very thin,fill,color=yellow] (-27,60) -- +(0,1) -- (5,29) -- +(-1,0) -- cycle;
  \draw[very thin,fill,color=green] (4,29) -- +(1,0) -- (36,60) -- +(0,1) -- cycle;
  \draw[very thin,fill,color=red] (27,69) -- +(0,1) -- (36,79) -- +(0,-1) -- cycle;
  \draw[very thin,fill,color=red] (-9,123) -- +(0,1) -- (0,133) -- +(0,-1) -- cycle;
  \draw[very thin,fill,color=red] (-9,141) -- +(0,1) -- (23,174) -- +(1,0) -- cycle;
  \draw[very thin,fill,color=green] (54,143) -- +(0,1) -- (45,135) -- +(0,-1) -- cycle;
  \draw[very thin,fill,color=green] (54,125) -- +(0,1) -- (45,117) -- +(0,-1) -- cycle;
  \draw[very thin,fill,color=green] (74,87) -- +(1,0) -- (108,120) -- +(0,1) -- cycle;
  \draw[very thin,fill,color=red] (99,129) -- +(0,1) -- (108,139) -- +(0,-1) -- cycle;
  \draw[very thin,fill,color=red] (72,174) -- +(1,0) -- (44,145) -- +(-1,0) -- cycle;
  \draw[very thin,fill,color=red] (9,179) -- +(0,1) -- (18,189) -- +(0,-1) -- cycle;
  \draw[very thin,fill,color=red] (9,197) -- +(0,1) -- (43,232) -- +(1,0) -- cycle;
  \draw[very thin,fill,color=green] (72,203) -- +(1,0) -- (108,238) -- +(0,1) -- cycle;
  \draw[very thin,fill,color=red] (99,247) -- +(0,1) -- (108,257) -- +(0,-1) -- cycle;
  \draw[very thin,fill,color=red] (74,290) -- +(1,0) -- (45,260) -- +(0,1) -- cycle;
  \draw[very thin,fill,color=green] (54,251) -- +(0,1) -- (45,243) -- +(0,-1) -- cycle;
  \draw[very thin,fill,color=green] (54,233) -- +(0,1) -- (23,203) -- +(1,0) -- cycle;
  \draw[very thin,fill,color=red] (-9,235) -- +(0,1) -- (0,245) -- +(0,-1) -- cycle;
  \draw[very thin,fill,color=red] (-9,253) -- +(0,1) -- (36,299) -- +(0,-1) -- cycle;
  \draw[very thin,fill,color=red] (27,307) -- +(0,1) -- (36,317) -- +(0,-1) -- cycle;
  \draw[very thin,fill,color=red] (4,348) -- +(1,0) -- (-27,316) -- +(0,1) -- cycle;
  \draw[very thin,fill,color=green] (-18,307) -- +(0,1) -- (-27,299) -- +(0,-1) -- cycle;
  \draw[very thin,fill,color=green] (18,253) -- +(0,1) -- (9,245) -- +(0,-1) -- cycle;
  \draw[very thin,fill,color=green] (18,235) -- +(0,1) -- (-15,203) -- +(1,0) -- cycle;
  \draw[very thin,fill,color=red] (-45,233) -- +(0,1) -- (-36,243) -- +(0,-1) -- cycle;
  \draw[very thin,fill,color=red] (-45,251) -- +(0,1) -- (-36,261) -- +(0,-1) -- cycle;
  \draw[very thin,fill,color=red] (-66,290) -- +(1,0) -- (-99,256) -- +(0,1) -- cycle;
  \draw[very thin,fill,color=green] (-90,247) -- +(0,1) -- (-99,239) -- +(0,-1) -- cycle;
  \draw[very thin,fill,color=green] (-64,203) -- +(1,0) -- (-34,232) -- +(-1,0) -- cycle;
  \draw[very thin,fill,color=green] (0,197) -- +(0,1) -- (-9,189) -- +(0,-1) -- cycle;
  \draw[very thin,fill,color=green] (0,179) -- +(0,1) -- (-35,145) -- +(1,0) -- cycle;
  \draw[very thin,fill,color=red] (-64,174) -- +(1,0) -- (-99,138) -- +(0,1) -- cycle;
  \draw[very thin,fill,color=green] (-90,129) -- +(0,1) -- (-99,121) -- +(0,-1) -- cycle;
  \draw[very thin,fill,color=green] (-66,87) -- +(1,0) -- (-36,116) -- +(0,1) -- cycle;
  \draw[very thin,fill,color=red] (-45,125) -- +(0,1) -- (-36,135) -- +(0,-1) -- cycle;
  \draw[very thin,fill,color=red] (-45,143) -- +(0,1) -- (-15,174) -- +(1,0) -- cycle;
  \draw[very thin,fill,color=green] (18,141) -- +(0,1) -- (9,133) -- +(0,-1) -- cycle;
  \draw[very thin,fill,color=green] (18,123) -- +(0,1) -- (-27,79) -- +(0,-1) -- cycle;
  \draw[very thin,fill,color=green] (-18,69) -- +(0,1) -- (-27,61) -- +(0,-1) -- cycle;
\end{tikzpicture}
\end{minipage}
\caption{Obstacle size $1/2 \times 1/2$. Periodic trajectories: 
slopes $1$, $3/7$, $7/9$, $9/11$ and $9/29$.}
\label{fig:periodic}
\end{figure}

\begin{figure}[ht]
\begin{minipage}[ht]{0.38\linewidth}
\centering
\begin{tikzpicture}
  % windtree, a = b = 1/2, slope = 3/4
  [x=0.583333333333cm,y=0.4375cm,scale=0.3]
  \fill [gray] (0,0) rectangle ++(3,4);
  \fill [gray] (12,0) rectangle ++(3,4);
  \foreach \i in {0,6,...,12}
  \foreach \j in {0,8,...,16}
    \draw[thin] (\i,\j) rectangle +(3,4);
  \draw[very thin,fill,color=blue] (6,8) -- +(0,1) -- (3,12) -- +(0,-1) -- cycle;
  \draw[very thin,fill,color=yellow] (7,16) -- +(1,0) -- (12,12) -- +(0,-1) -- cycle;
  \draw[very thin,fill,color=yellow] (9,8) -- +(0,1) -- (14,4) -- +(-1,0) -- cycle;
  \draw[very thin,fill,color=green] (1,4) -- +(1,0) -- (6,8) -- +(0,1) -- cycle;
  \draw[very thin,fill,color=red] (3,11) -- +(0,1) -- (7,16) -- +(1,0) -- cycle;
  \draw[very thin,fill,color=green] (12,11) -- +(0,1) -- (9,9) -- +(0,-1) -- cycle;
\end{tikzpicture}
\end{minipage}%
\begin{minipage}[ht]{0.58\linewidth}
\centering
\begin{tikzpicture}
  % windtree, a = b = 1/2, slope = 10/17
  [x=0.0972222222222cm,y=0.0571895424837cm,scale=0.4]
  \fill [gray] (0,0) rectangle ++(10,17);
  \fill [gray] (140,102) rectangle ++(10,17);
  \foreach \i in {0,20,...,140}
  \foreach \j in {-68,-34,...,136}
    \draw[thin] (\i,\j) rectangle +(10,17);
  \draw[very thin,fill,color=yellow] (22,34) -- +(1,0) -- (40,17) -- +(0,-1) -- cycle;
  \draw[very thin,fill,color=yellow] (30,6) -- +(0,1) -- (60,-23) -- +(0,-1) -- cycle;
  \draw[very thin,fill,color=yellow] (50,-34) -- +(0,1) -- (68,-51) -- +(-1,0) -- cycle;
  \draw[very thin,fill,color=yellow] (84,-34) -- +(1,0) -- (102,-51) -- +(-1,0) -- cycle;
  \draw[very thin,fill,color=blue] (120,-33) -- +(0,1) -- (110,-22) -- +(0,-1) -- cycle;
  \draw[very thin,fill,color=blue] (140,7) -- +(0,1) -- (110,38) -- +(0,-1) -- cycle;
  \draw[very thin,fill,color=blue] (120,47) -- +(0,1) -- (90,78) -- +(0,-1) -- cycle;
  \draw[very thin,fill,color=blue] (120,107) -- +(0,1) -- (110,118) -- +(0,-1) -- cycle;
  \draw[very thin,fill,color=yellow] (128,136) -- +(1,0) -- (146,119) -- +(-1,0) -- cycle;
  \draw[very thin,fill,color=green] (5,17) -- +(1,0) -- (23,34) -- +(-1,0) -- cycle;
  \draw[very thin,fill,color=green] (40,16) -- +(0,1) -- (30,7) -- +(0,-1) -- cycle;
  \draw[very thin,fill,color=green] (60,-24) -- +(0,1) -- (50,-33) -- +(0,-1) -- cycle;
  \draw[very thin,fill,color=green] (67,-51) -- +(1,0) -- (85,-34) -- +(-1,0) -- cycle;
  \draw[very thin,fill,color=green] (101,-51) -- +(1,0) -- (120,-33) -- +(0,1) -- cycle;
  \draw[very thin,fill,color=red] (110,-23) -- +(0,1) -- (140,8) -- +(0,-1) -- cycle;
  \draw[very thin,fill,color=red] (110,37) -- +(0,1) -- (120,48) -- +(0,-1) -- cycle;
  \draw[very thin,fill,color=red] (90,77) -- +(0,1) -- (120,108) -- +(0,-1) -- cycle;
  \draw[very thin,fill,color=red] (110,117) -- +(0,1) -- (128,136) -- +(1,0) -- cycle;
\end{tikzpicture}
\end{minipage}%

\bigskip
\bigskip
\medskip

\begin{minipage}[ht]{0.38\linewidth}
\centering
\begin{tikzpicture}
  % windtree, a = b = 1/2, slope = 14/17
  [x=0.078125cm,y=0.0643382352941cm,scale=0.4]
  \fill [gray] (0,0) rectangle ++(14,17);
  \fill [gray] (28,68) rectangle ++(14,17);
  \foreach \i in {-28,0,...,56}
  \foreach \j in {0,34,...,170}
    \draw[thin] (\i,\j) rectangle +(14,17);
  \draw[very thin,fill,color=blue] (28,37) -- +(0,1) -- (-14,80) -- +(0,-1) -- cycle;
  \draw[very thin,fill,color=yellow] (8,102) -- +(1,0) -- (28,83) -- +(0,-1) -- cycle;
  \draw[very thin,fill,color=yellow] (14,68) -- +(0,1) -- (32,51) -- +(-1,0) -- cycle;
  \draw[very thin,fill,color=blue] (56,75) -- +(0,1) -- (30,102) -- +(-1,0) -- cycle;
  \draw[very thin,fill,color=blue] (12,85) -- +(1,0) -- (-14,112) -- +(0,-1) -- cycle;
  \draw[very thin,fill,color=yellow] (10,136) -- +(1,0) -- (28,119) -- +(0,-1) -- cycle;
  \draw[very thin,fill,color=yellow] (14,104) -- +(0,1) -- (34,85) -- +(-1,0) -- cycle;
  \draw[very thin,fill,color=blue] (56,107) -- +(0,1) -- (14,150) -- +(0,-1) -- cycle;
  \draw[very thin,fill,color=yellow] (34,170) -- +(1,0) -- (56,149) -- +(0,-1) -- cycle;
  \draw[very thin,fill,color=yellow] (14,106) -- +(0,1) -- (36,85) -- +(-1,0) -- cycle;
  \draw[very thin,fill,color=green] (7,17) -- +(1,0) -- (28,37) -- +(0,1) -- cycle;
  \draw[very thin,fill,color=red] (-14,79) -- +(0,1) -- (8,102) -- +(1,0) -- cycle;
  \draw[very thin,fill,color=green] (28,82) -- +(0,1) -- (14,69) -- +(0,-1) -- cycle;
  \draw[very thin,fill,color=green] (31,51) -- +(1,0) -- (56,75) -- +(0,1) -- cycle;
  \draw[very thin,fill,color=red] (29,102) -- +(1,0) -- (13,85) -- +(-1,0) -- cycle;
  \draw[very thin,fill,color=red] (-14,111) -- +(0,1) -- (10,136) -- +(1,0) -- cycle;
  \draw[very thin,fill,color=green] (28,118) -- +(0,1) -- (14,105) -- +(0,-1) -- cycle;
  \draw[very thin,fill,color=green] (33,85) -- +(1,0) -- (56,107) -- +(0,1) -- cycle;
  \draw[very thin,fill,color=red] (14,149) -- +(0,1) -- (34,170) -- +(1,0) -- cycle;
  \draw[very thin,fill,color=green] (56,148) -- +(0,1) -- (14,107) -- +(0,-1) -- cycle;
\end{tikzpicture}
\end{minipage}%
\begin{minipage}[ht]{0.58\linewidth}
\centering
\begin{tikzpicture}
  % windtree, a = b = 1/2, slope = 16/37
  [x=0.0546875cm,y=0.0236486486486cm,scale=0.4]
  \fill [gray] (0,0) rectangle ++(16,37);
  \fill [gray] (128,-74) rectangle ++(16,37);
  \foreach \i in {-160,-128,...,128}
  \foreach \j in {-444,-370,...,74}
    \draw[thin] (\i,\j) rectangle +(16,37);
  \draw[very thin,fill,color=yellow] (45,74) -- +(1,0) -- (96,24) -- +(0,-1) -- cycle;
  \draw[very thin,fill,color=yellow] (80,7) -- +(0,1) -- (128,-40) -- +(0,-1) -- cycle;
  \draw[very thin,fill,color=yellow] (112,-57) -- +(0,1) -- (128,-72) -- +(0,-1) -- cycle;
  \draw[very thin,fill,color=yellow] (80,-121) -- +(0,1) -- (96,-136) -- +(0,-1) -- cycle;
  \draw[very thin,fill,color=blue] (47,-185) -- +(1,0) -- (11,-148) -- +(-1,0) -- cycle;
  \draw[very thin,fill,color=blue] (-27,-185) -- +(1,0) -- (-63,-148) -- +(-1,0) -- cycle;
  \draw[very thin,fill,color=yellow] (-112,-197) -- +(0,1) -- (-96,-212) -- +(0,-1) -- cycle;
  \draw[very thin,fill,color=yellow] (-144,-261) -- +(0,1) -- (-128,-276) -- +(0,-1) -- cycle;
  \draw[very thin,fill,color=yellow] (-144,-293) -- +(0,1) -- (-96,-340) -- +(0,-1) -- cycle;
  \draw[very thin,fill,color=yellow] (-112,-357) -- +(0,1) -- (-61,-407) -- +(-1,0) -- cycle;
  \draw[very thin,fill,color=yellow] (-25,-370) -- +(1,0) -- (13,-407) -- +(-1,0) -- cycle;
  \draw[very thin,fill,color=blue] (64,-356) -- +(0,1) -- (48,-339) -- +(0,-1) -- cycle;
  \draw[very thin,fill,color=blue] (96,-292) -- +(0,1) -- (80,-275) -- +(0,-1) -- cycle;
  \draw[very thin,fill,color=blue] (96,-260) -- +(0,1) -- (48,-211) -- +(0,-1) -- cycle;
  \draw[very thin,fill,color=blue] (64,-196) -- +(0,1) -- (16,-147) -- +(0,-1) -- cycle;
  \draw[very thin,fill,color=blue] (32,-132) -- +(0,1) -- (16,-115) -- +(0,-1) -- cycle;
  \draw[very thin,fill,color=blue] (64,-68) -- +(0,1) -- (48,-51) -- +(0,-1) -- cycle;
  \draw[very thin,fill,color=yellow] (99,0) -- +(1,0) -- (137,-37) -- +(-1,0) -- cycle;
  \draw[very thin,fill,color=green] (8,37) -- +(1,0) -- (46,74) -- +(-1,0) -- cycle;
  \draw[very thin,fill,color=green] (96,23) -- +(0,1) -- (80,8) -- +(0,-1) -- cycle;
  \draw[very thin,fill,color=green] (128,-41) -- +(0,1) -- (112,-56) -- +(0,-1) -- cycle;
  \draw[very thin,fill,color=green] (128,-73) -- +(0,1) -- (80,-120) -- +(0,-1) -- cycle;
  \draw[very thin,fill,color=green] (96,-137) -- +(0,1) -- (47,-185) -- +(1,0) -- cycle;
  \draw[very thin,fill,color=red] (10,-148) -- +(1,0) -- (-26,-185) -- +(-1,0) -- cycle;
  \draw[very thin,fill,color=red] (-64,-148) -- +(1,0) -- (-112,-197) -- +(0,1) -- cycle;
  \draw[very thin,fill,color=green] (-96,-213) -- +(0,1) -- (-144,-260) -- +(0,-1) -- cycle;
  \draw[very thin,fill,color=green] (-128,-277) -- +(0,1) -- (-144,-292) -- +(0,-1) -- cycle;
  \draw[very thin,fill,color=green] (-96,-341) -- +(0,1) -- (-112,-356) -- +(0,-1) -- cycle;
  \draw[very thin,fill,color=green] (-62,-407) -- +(1,0) -- (-24,-370) -- +(-1,0) -- cycle;
  \draw[very thin,fill,color=green] (12,-407) -- +(1,0) -- (64,-356) -- +(0,1) -- cycle;
  \draw[very thin,fill,color=red] (48,-340) -- +(0,1) -- (96,-291) -- +(0,-1) -- cycle;
  \draw[very thin,fill,color=red] (80,-276) -- +(0,1) -- (96,-259) -- +(0,-1) -- cycle;
  \draw[very thin,fill,color=red] (48,-212) -- +(0,1) -- (64,-195) -- +(0,-1) -- cycle;
  \draw[very thin,fill,color=red] (16,-148) -- +(0,1) -- (32,-131) -- +(0,-1) -- cycle;
  \draw[very thin,fill,color=red] (16,-116) -- +(0,1) -- (64,-67) -- +(0,-1) -- cycle;
  \draw[very thin,fill,color=red] (48,-52) -- +(0,1) -- (99,0) -- +(1,0) -- cycle;
\end{tikzpicture}
\end{minipage}

\bigskip
\bigskip
\medskip

\begin{minipage}[ht]{0.98\linewidth}
\centering
\begin{tikzpicture}
  % windtree, a = b = 1/2, slope = 16/39
  [x=0.0320601851852cm,y=0.0131528964862cm,scale=0.5]
  \fill [gray] (0,0) rectangle ++(16,39);
  \fill [gray] (736,468) rectangle ++(16,37);
  \foreach \i in {0,32,...,736}
  \foreach \j in {-156,-78,...,546}
    \draw[very thin] (\i,\j) rectangle +(16,39);
  \draw[very thin,fill,color=yellow] (47,78) -- +(1,0) -- (96,30) -- +(0,-1) -- cycle;
  \draw[very thin,fill,color=yellow] (80,13) -- +(0,1) -- (133,-39) -- +(-1,0) -- cycle;
  \draw[very thin,fill,color=yellow] (171,0) -- +(1,0) -- (224,-52) -- +(0,-1) -- cycle;
  \draw[very thin,fill,color=yellow] (208,-69) -- +(0,1) -- (257,-117) -- +(-1,0) -- cycle;
  \draw[very thin,fill,color=yellow] (295,-78) -- +(1,0) -- (335,-117) -- +(-1,0) -- cycle;
  \draw[very thin,fill,color=blue] (384,-68) -- +(0,1) -- (368,-51) -- +(0,-1) -- cycle;
  \draw[very thin,fill,color=yellow] (419,0) -- +(1,0) -- (459,-39) -- +(-1,0) -- cycle;
  \draw[very thin,fill,color=blue] (512,14) -- +(0,1) -- (496,31) -- +(0,-1) -- cycle;
  \draw[very thin,fill,color=blue] (544,78) -- +(0,1) -- (528,95) -- +(0,-1) -- cycle;
  \draw[very thin,fill,color=blue] (544,110) -- +(0,1) -- (496,159) -- +(0,-1) -- cycle;
  \draw[very thin,fill,color=blue] (512,174) -- +(0,1) -- (496,191) -- +(0,-1) -- cycle;
  \draw[very thin,fill,color=blue] (544,238) -- +(0,1) -- (528,255) -- +(0,-1) -- cycle;
  \draw[very thin,fill,color=blue] (544,270) -- +(0,1) -- (496,319) -- +(0,-1) -- cycle;
  \draw[very thin,fill,color=blue] (512,334) -- +(0,1) -- (496,351) -- +(0,-1) -- cycle;
  \draw[very thin,fill,color=blue] (544,398) -- +(0,1) -- (528,415) -- +(0,-1) -- cycle;
  \draw[very thin,fill,color=yellow] (581,468) -- +(1,0) -- (621,429) -- +(-1,0) -- cycle;
  \draw[very thin,fill,color=blue] (672,480) -- +(0,1) -- (656,497) -- +(0,-1) -- cycle;
  \draw[very thin,fill,color=yellow] (705,546) -- +(1,0) -- (745,507) -- +(-1,0) -- cycle;
  \draw[very thin,fill,color=green] (8,39) -- +(1,0) -- (48,78) -- +(-1,0) -- cycle;
  \draw[very thin,fill,color=green] (96,29) -- +(0,1) -- (80,14) -- +(0,-1) -- cycle;
  \draw[very thin,fill,color=green] (132,-39) -- +(1,0) -- (172,0) -- +(-1,0) -- cycle;
  \draw[very thin,fill,color=green] (224,-53) -- +(0,1) -- (208,-68) -- +(0,-1) -- cycle;
  \draw[very thin,fill,color=green] (256,-117) -- +(1,0) -- (296,-78) -- +(-1,0) -- cycle;
  \draw[very thin,fill,color=green] (334,-117) -- +(1,0) -- (384,-68) -- +(0,1) -- cycle;
  \draw[very thin,fill,color=red] (368,-52) -- +(0,1) -- (419,0) -- +(1,0) -- cycle;
  \draw[very thin,fill,color=green] (458,-39) -- +(1,0) -- (512,14) -- +(0,1) -- cycle;
  \draw[very thin,fill,color=red] (496,30) -- +(0,1) -- (544,79) -- +(0,-1) -- cycle;
  \draw[very thin,fill,color=red] (528,94) -- +(0,1) -- (544,111) -- +(0,-1) -- cycle;
  \draw[very thin,fill,color=red] (496,158) -- +(0,1) -- (512,175) -- +(0,-1) -- cycle;
  \draw[very thin,fill,color=red] (496,190) -- +(0,1) -- (544,239) -- +(0,-1) -- cycle;
  \draw[very thin,fill,color=red] (528,254) -- +(0,1) -- (544,271) -- +(0,-1) -- cycle;
  \draw[very thin,fill,color=red] (496,318) -- +(0,1) -- (512,335) -- +(0,-1) -- cycle;
  \draw[very thin,fill,color=red] (496,350) -- +(0,1) -- (544,399) -- +(0,-1) -- cycle;
  \draw[very thin,fill,color=red] (528,414) -- +(0,1) -- (581,468) -- +(1,0) -- cycle;
  \draw[very thin,fill,color=green] (620,429) -- +(1,0) -- (672,480) -- +(0,1) -- cycle;
  \draw[very thin,fill,color=red] (656,496) -- +(0,1) -- (705,546) -- +(1,0) -- cycle;
\end{tikzpicture}
\end{minipage}

\caption{Escaping trajectories: 
slopes $3/4$, $10/17$, $14/17$, $16/37$, and $16/39$, for obstacle size $1/2 
\times 1/2$. The grayed obstacles 
are hit at the same location, so the trajectory repeats afterwards 
with a drift.}
\label{fig:escaping}
\end{figure}
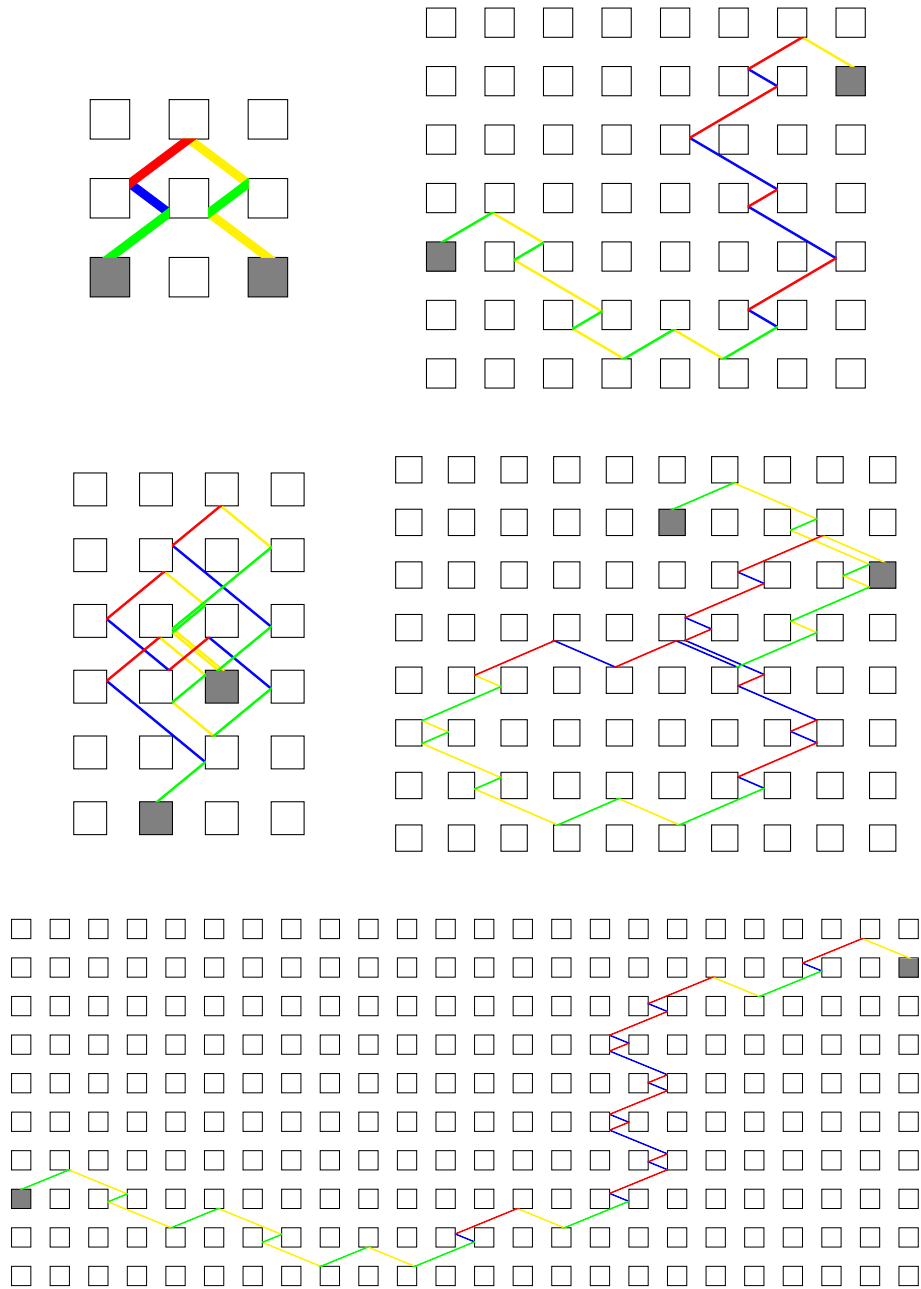


\newpage

\begin{thebibliography}{EMM2}

\bibitem[Bo]{Bo} J. Bowman,
\textit{Flat structures and complex structures in Teichm\"uller theory}, Ph. D. thesis, Cornell University, (2009).

\bibitem[DDL]{DDL}
M.\ Degli Esposti, G.\ Del Magno, M.\ Lenci, \textit{Escape orbits and
ergodicity in infinite step billiards,} Nonlinearity 13 (2000),
1275--1292.

\bibitem[EhEh]{EhEh} P.\ and T.\ Ehrenfest, \textit{Begriffliche Grundlagen 
der statistischen Auffassung in der Mechanik} 
Encykl. d. Math. Wissensch. IV 2 II, Heft 6, 90 S (1912)
(in German, translated in:)
\textit{The conceptual foundations of the
statistical approach in mechanics,} (trans. Moravicsik,
M. J.), 10-13 Cornell University Press, Itacha NY, (1959).

\bibitem[GuTr]{GuTr} E.~Gutkin and S.~Troubetzkoy, \textit{Directional flows and
strong recurrence for polygonal billiards,} in Proceedings of the International
Congress of Dynamical Systems, Montevideo, Uruguay, F.~Ledrappier et al.~eds.
(1996) 21--45.

\bibitem[HaWe]{HaWe} J. Hardy, J. Weber, \textit{Diffusion in a periodic wind-tree model}, J. Math. Phys. 21 (7),(1980) 1802--1808.

\bibitem[Ho]{Ho} W. P. Hooper, \textit{Dynamics on an infinite surface with the lattice property,} (2007) preprint.  

\bibitem[HuLe]{HuLe} P. Hubert, S. Lelièvre, \textit{Prime 
arithmetic Teichmüller discs in $\HHH(2)$,} Israel J. Math. 151 (2006), 281--321.

\bibitem[HuSc]{HuSc} P.\ Hubert, G. Schmithüsen, 
\textit{Infinite translation surfaces with infinitely generated Veech
  groups} (2009) preprint.

\bibitem[HuWe]{HuWe} P. Hubert, B. Weiss, \textit{Dynamics on the infinite staircase,} (2008) preprint.

\bibitem[Ka1]{Ka1} E. Kani, \textit{Hurwitz spaces of genus 2 covers of an elliptic curve,} Collect. Math. 54 (2003), no. 1, 1--51.

\bibitem[Ka2]{Ka2} E. Kani, \textit{The number of genus 2 covers of an elliptic curve,} Manuscripta Math. 121 (2006), no. 1, 51--80.

\bibitem[MaTa]{MaTa} H. Masur, S. Tabachnikov, 
\textit{Rational billiards and flat structures},  Handbook of dynamical systems, Vol. 1A, 1015--1089, North-Holland, Amsterdam, 2002.

\bibitem[Mc]{Mc} C. McMullen, \textit{Teichmüller curves in genus two:
    discriminant and spin}, Math. Ann. 333 no. 1, (2005)  87--130.
    
\bibitem[Su]{Su} D. Sullivan, \textit{Disjoint spheres, approximation by imaginary quadratic numbers, and the 
logarithm law for geodesics,} Acta Math., 149 (1982) 215-237.     

\bibitem[Tr]{Tr} S.\ Troubetzkoy, \textit{Billiards in infinite
    polygons,}  Nonlinearity  12  (1999)  513--524.


\bibitem[Tr1]{Tr1} S.~Troubetzkoy, \textit{Recurrence and periodic billiard orbits in
polygons,} Regul.~Chaotic Dyn.~9  (2004)  1--12.


\bibitem[Tr2]{Tr2} S.~Troubetzkoy, \textit{Periodic billiard orbits in right
triangles}, Annales de l'Inst. Fourier 55  (2005)  29--46. 

\bibitem[Ve]{Ve} W. A. Veech, \textit{Teichmüller curves in moduli space, Eisenstein series and an application to triangular billiards}, Invent. Math. {\bf 97} (1989), no. 3, 553--583.

\bibitem[Vi]{Vi} M.~Viana, \textit{Dynamics of interval exchange maps and Teichmüller flows}, preprint.

\bibitem[Zo]{Zo} A. Zorich, \textit{Flat surfaces}, Frontiers in number theory, physics, and geometry I, 437--583, Springer, Berlin, 2006.
  
\end{thebibliography}
\end{document}